\documentclass[oneside,11pt]{amsart}
\usepackage{amsmath}
\usepackage{amssymb}
\usepackage{graphicx, overpic}
\usepackage{caption}
\usepackage{subcaption}
\usepackage{color}

\usepackage[margin=1.14in]{geometry}

\newtheorem{thm}{Theorem}[section]
\newtheorem{prop}[thm]{Proposition}
\newtheorem{lem}[thm]{Lemma}
\newtheorem{cor}[thm]{Corollary}

\theoremstyle{definition}
\newtheorem{defn}[thm]{Definition}
\newtheorem{rem}[thm]{Remark}

\newtheorem{cons}[thm]{Construction}
\newtheorem{notation}[thm]{Notation}

\renewcommand{\bar}[1]{\overline{#1}}
\newcommand{\boundary}{\partial}

\newcommand{\presentation}[2]{\langle\, {#1} \mid {#2} \,\rangle}
\newcommand{\bigpresentation}[2]{ \bigl\langle \, {#1} \bigm| {#2} \,
                                \bigr\rangle }
\newcommand{\set}[2]{\{\,{#1} \mid {#2} \,\}}

\renewcommand{\emptyset}{\varnothing}

\newcommand{\field}[1]{\mathbb{#1}}
\newcommand{\Z}{\field{Z}}

\newcommand{\R}{\field{R}}

\newcommand{\N}{\field{N}}
\newcommand{\E}{\field{E}}

\newcommand{\RR}{\field{R}}
\newcommand{\XX}{\field{X}}
\newcommand{\cC}{\mathcal{C}}
\newcommand{\cW}{\mathcal{W}}

\newcommand{\cO}{\mathcal{O}}

\DeclareMathOperator{\CAT}{CAT}

\DeclareMathOperator{\Image}{Im}

\newcommand{\la}{\left\langle}
\newcommand{\ra}{\right\rangle}

\definecolor{amethyst}{rgb}{0.6, 0.4, 0.8}

\newcommand{\hide}[1]{}

\usepackage{combelow}
\newcommand{\Drutu}{Dru{\cb{t}}u}

\begin{document}

\title[Surface group amalgams]{Surface group amalgams that (don't) act on $3$--manifolds}

\author{G. Christopher Hruska}
\author{Emily Stark}
\author{Hung Cong Tran}
\address{Department of Mathematical Sciences\\
University of Wisconsin--Milwaukee\\
P.O.~Box 413\\
Milwaukee, WI 53201\\
USA}
\email{chruska@uwm.edu}
\address{Department of Mathematics\\
Technion - Israel Institute of Technology \\
Haifa 32000 \\
Israel }
\email{emily.stark@technion.ac.il}
\address{Department of Mathematics\\
The University of Georgia\\
1023 D.W. Brooks Drive\\
Athens, GA 30605\\
USA}
\email{hung.tran@uga.edu}

\date{\today}

\begin{abstract}
  We determine which amalgamated products of surface groups identified over multiples of simple closed curves are not fundamental groups of $3$--manifolds. We prove each surface amalgam considered is virtually the fundamental group of a $3$--manifold.  We prove that each such surface group amalgam is abstractly commensurable to a right-angled Coxeter group from a related family. In an appendix, we determine the quasi-isometry classes among these surface amalgams and their related right-angled Coxeter groups. 
\end{abstract}

\maketitle

\section{Introduction}
  
  Surface groups have a foundational role in geometric topology and geometric group theory. In this paper, we study a class of amalgamated products of surface groups, which in many ways resemble fundamental groups of $3$--manifolds. Let $\mathcal{C}_{m,n}$ be the collection of amalgamated free products of the form $\pi_1(S_g)*_{\langle a^m=b^n \rangle}\pi_1(S_h)$, where $m \leq n$, and $S_g$ and $S_h$ are closed orientable surfaces of genus $g$ and~$h$ greater than one, and $a$ and $b$ are the homotopy class of an essential simple closed curve on $S_g$ and~$S_h$, respectively. Let 
  $\cC = \bigcup_{m\leq n} \cC_{m,n}$.
  Each group in $\cC$ is the fundamental group of a complex that consists of two closed orientable surfaces $S_g$ and $S_h$ and an annulus, where one boundary component of the annulus is identified to an essential simple closed curve $a$ on $S_g$ by a degree--$m$ map of the circle, and the other boundary component of the annulus is identified to an essential simple closed curve $b$ on $S_h$ by a degree--$n$ map of the circle for positive integers $m \leq n$. An example appears in Figure~\ref{figure:surf_amalgam}.

A primary feature of the surface amalgams in $\cC$ is their strong resemblance in many ways to $3$--manifolds with three different types of geometric decomposition.
More specifically, $\cC$ is divided into three basic families of amalgams.  Amalgams in the first family are word hyperbolic and resemble Kleinian groups.  The amalgams in the second family are hyperbolic relative to virtually abelian subgroups (Klein bottle groups), and strongly resemble $3$--manifolds formed by gluing two hyperbolic components along a cusp torus.
Amalgams in the third family are hyperbolic relative to groups of the form $\presentation{a,b}{a^m = b^n}$, which has a finite-index subgroup isomorphic to $F\times \Z$ for a nonabelian free group $F$.  We note that when $m$ and $m$ are relatively prime, the group $\presentation{a,b}{a^m = b^n}$ is the fundamental group of a torus knot complement in $S^3$.
Surface amalgams of this third family closely resemble ``mixed'' type $3$--manifolds that contain both hyperbolic and Seifert fibered JSJ components. 

However, in spite of their strong resemblance to various $3$--manifolds, 
we prove, for most choices of $m$ and $n$, the surface amalgams in $\cC$ are not the fundamental groups of $3$--manifolds.
   
  \begin{thm} \label{sec1_3man_classification}
  Let $G \cong \pi_1(S_g) *_{\la a^m=b^n \ra} \pi_1(S_h) \in \mathcal{C}_{m,n}$, where $a\in \pi_1(S_g)$ and $b \in \pi_1(S_h)$ are homotopy classes of essential simple closed curves. Then $G$ is the fundamental group of a $3$--manifold if and only if one of the following holds: 
  \begin{enumerate}
   \item $m=n=1$;
   \item $m=1$, $n=2$, and $b$ is the homotopy class of a non-separating curve; or,
   \item $m=n=2$, and $a$ and $b$ are homotopy classes of non-separating curves. 
  \end{enumerate}
 \end{thm}
  
      \begin{figure}
      \begin{overpic}[scale=.7, tics=5]{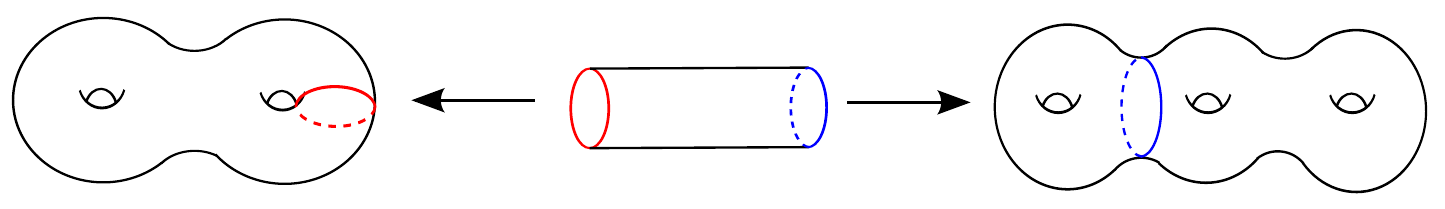}
	\put(22,9){\small{$a$}}
	\put(32,9){\small{$a^m$}}
	\put(61,9){\small{$b^n$}}
	\put(75.5,8.5){\small{$b$}}
	\end{overpic}
	\caption{{\small  A $2$--complex whose fundamental group is a surface amalgam in the family $\cC_{m,n}$. The left boundary curve of the annulus is glued to the curve $a^m$, and the right boundary curve of the annulus is glued to the curve $b^n$.}}
      \label{figure:surf_amalgam}
     \end{figure}
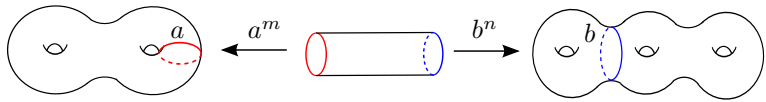
  
  On the other hand, we prove all groups in $\cC$ are virtually $3$--manifold groups. 
  
  \begin{thm} \label{sec1:virt3man}
   Each surface amalgam in the family $\cC$ has a finite-index subgroup that is the fundamental group of a $3$--manifold. 
  \end{thm}

  Kapovich--Kleiner \cite{kapovichkleiner} introduced the first examples of torsion-free hyperbolic groups that are not $3$--manifold groups but are virtual $3$--manifold groups. 
  In their later work on coarse Poincar\'{e} duality spaces \cite{kapovichkleiner05}, they observe that many ``higher genus Baumslag--Solitar groups'' are also not $3$--manifold groups. 
Higher genus Baumslag--Solitar groups are certain HNN-extensions of hyperbolic surface groups with two boundary components over powers of the boundary curves.
In this paper, we extend and generalize the Kapovich--Kleiner construction to a broad family of surface group amalgams.

In the present setting of surface amalgams over powers of embedded curves in closed surfaces, the combinatorial obstruction to acting on a $3$--manifold is substantially simpler than the obstruction found by Kapovich--Kleiner.  Ultimately, the obstruction we found boils down to a very elementary argument about subgroups of finite dihedral groups (see for example ``Case~1'' in the proof of Theorem~\ref{thm:3manclassification}).
Moreover, some interesting features arise when considering closed surfaces, which contain both separating and nonseparating curves. 
In contrast, the Kapovich--Kleiner examples are formed by gluing over powers of boundary curves, which are homologically trivial.

  There has been much interest in geometric group theory recently in the problem of embedding groups as quasi-convex subgroups of right-angled Artin groups (see for example \cite{wise}). 
Some of the strongest consequences follow from the fact that right-angled Artin groups embed in right-angled Coxeter groups, which are subgroups of $SL_n(\Z)$.  In particular, every virtually special cubulated group is 
linear over $\Z$ because of its embedding in a right-angled Coxeter group \cite{HsuWise99,DavisJanuszkiewicz00}.
In this paper, we obtain a strong embedding theorem; we prove each surface amalgam in $\cC$ virtually embeds as a finite-index subgroup in  a particularly simple right-angled Coxeter group, one whose nerve is a planar graph called a {\it generalized $\Theta$-graph} (see Definition~\ref{def:RACG}).

\begin{thm}
\label{thm:sec1_CW_AC}
Each surface amalgam in $\cC$ is abstractly commensurable to a right-angled Coxeter group with nerve a generalized $\Theta$-graph.
\end{thm}
    
  By an elementary construction of Davis--Okun \cite{davisokun}, each right-angled Coxeter group with nerve a planar, connected simplicial complex of dimension at most two acts properly on a contractible $3$--manifold; see Theorem~\ref{thm:racg_3man} for more details. 
  Thus, Theorem~\ref{sec1:virt3man} follows from Theorem~\ref{thm:sec1_CW_AC}.   
  The class $\cW$ of right-angled Coxeter with generalized $\Theta$-graph nerve was introduced by Dani--Thomas \cite{danithomas}, as a generalization of a family of groups investigated by Crisp--Paoluzzi \cite{crisppaoluzzi}.

  A natural problem in geometric group theory proposed by Gromov is to characterize the quasi-isometry classes within a class of finitely-generated groups. Recent work of Cashen--Martin \cite{cashenmartin} provides a far-reaching quasi-isometry classification for groups which split over $2$-ended subgroups using highly intricate combinatorial machinery. In an appendix, we provide the quasi-isometry classification of all groups in the families $\cC$ and $\cW$.
The classification obtained in the appendix is a special case of the general classification theorem of Cashen--Martin. In particular, within the family $\cC$ we obtain a simple statement of the quasi-isometry classification in terms of the degree of the gluing maps in the amalgams. 
  
  \begin{thm} \label{thm:sec1_C_QI}
   \it Let $G \in \mathcal{C}_{m,n}$ and $G' \in \mathcal{C}_{m',n'}$. Then $G$ and $G'$ are quasi-isometric if and only if one of the following conditions hold:
  \begin{enumerate}
  \item $m=m'=1$ and $n=n'$; or, 
  \item $m=m'=n=n'=2$; or, 
  \item $m\geq 2, n\geq 3, m'\geq 2$ and $n'\geq 3$. 
  \end{enumerate}
  \end{thm}
  
Although the conclusion of Theorem~\ref{thm:sec1_C_QI} may not be surprising to experts, we have included it as an illustration for a less-expert reader of some simple ways that the powerful tools of Behrstock--Neumann \cite{behrstockneumann} and Whyte \cite{whyte} may be used to produce quasi-isometries between trees of spaces whose vertex spaces have ``treelike'' geometries.  Furthermore, we relate the geometry of surface amalgams in the family $\cC$ to the right-angled Coxeter groups in the family $\cW$. Dani--Thomas provide a quasi-isometry classification of the hyperbolic groups of $\cW$, and we extend this classification to cover the non-hyperbolic groups in $\cW$ in Theorem~\ref{sec1_thm_cW_QI}. 
 
  A surface amalgam in $\cC_{m,n}$ is $\delta$--hyperbolic if and only if $m=1$. Thus, by Theorem~\ref{thm:sec1_C_QI} there are infinitely many quasi-isometry classes among the $\delta$--hyperbolic groups in $\cC$ and exactly two quasi-isometry classes among the non-hyperbolic groups in $\cC$. The quasi-isometry classification of $\delta$--hyperbolic groups in $\cC$ follows in spirit from the quasi-isometry classification of geometric amalgams of free groups given in the thesis of Malone \cite{malone}. A proof of the quasi-isometry classification in the hyperbolic case that $m=n=1$ is given by Stark \cite{stark},
  and the result in the hyperbolic setting also follows from Cashen--Martin \cite{cashenmartin}. The non-hyperbolic case above would also follow by translating \cite[Theorem~8.5]{cashenmartin} to this setting.

  \subsection{Methods of proof} \label{subsec:method}

 \subsubsection{$3$--manifold groups}
 
In Theorem~\ref{sec1_3man_classification} we determine which surface amalgams in the family $\cC$ are fundamental groups of $3$--manifolds. The proof uses results due to Kapovich--Kleiner \cite{kapovichkleiner05} on coarse Poincar\'{e} duality spaces. In dimension~$3$, Kapovich--Kleiner prove (with additional mild hypotheses, see Section~\ref{sec:3man},) that if $W$ is a union of $k$ half-planes glued along their boundaries, then a uniformly proper embedding of~$W$ into a coarse $PD(3)$ space $X$ coarsely separates the space $X$ into $k$ deep components, and there is a cyclic order on the set of components that is preserved by any homeomorphism of $X$ stabilizing the image of~$W$.
  
If the conditions of Theorem~\ref{sec1_3man_classification} do not hold, we analyze the action of a group $G \in \cC$ on a model geometry for the group to find a union of $k$ half-planes whose $G$--stabilizer cannot preserve any cyclic order, which implies that $G$ cannot act properly on any coarse $PD(3)$ space. On the other hand, if the conditions of Theorem~\ref{sec1_3man_classification} hold for $G \in \cC$, we explicitly construct a $3$--manifold with fundamental group $G$ by gluing together an $I$--bundle over each surface along an annulus on the boundary of each $I$--bundle.  

By Theorem~\ref{sec1:virt3man} each group in $\cC$ is virtually the fundamental group of a $3$--manifold. Thus, by work of Bestvina--Kapovich--Kleiner \cite{bestvinakapovichkleiner} if $G \in \cC$ acts properly and cocompactly on a $\CAT(0)$ space $X$, then the visual boundary of $X$ does not contain an embedded non-planar graph.

Finally, observe the following, which contrasts with the classification theorem obtained in Theorem~\ref{sec1_3man_classification}.
A $\delta$--hyperbolic group in $\cC_{1,n}$ is quasi-isometric to (and, in fact, abstractly commensurable to) the fundamental group of the union of $2n+2$ hyperbolic surfaces with one boundary component identified to each other along their boundary curves. The fundamental group of this complex is also the fundamental group of a hyperbolic $3$--manifold with boundary called a book of $I$--bundles \cite{cullershalen}.  (See also \cite{haissinskypaoluzziwalsh} for more about these examples.)  In contrast, by Theorem~\ref{sec1_3man_classification}, most $\delta$--hyperbolic groups in $\cC$ are not fundamental groups of $3$--manifolds.

\subsubsection{Quasi-isometry classification}

  To study the quasi-isometry classes within $\cC$, we prove that each group in $\cC$ has a model geometry built from gluing together copies of a fattened tree in the sense of Behrstock--Neumann (see Definition~\ref{fattenedtree}) and copies of $T_{m,n} \times \R$, where $T_{m,n}$ denotes the biregular tree with vertices of valance $m$ and $n$ (see Definition~\ref{treedef}). We provide a quasi-isometry classification of such model geometries, and we use results of Whyte \cite{whyte} and Behrstock--Neumann \cite{behrstockneumann} on the geometry of the components of these spaces to construct a quasi-isometry. 
  
  To distinguish the quasi-isometry classes within these model geometries, we show that each right-angled Coxeter group with generalized $\Theta$--graph nerve, as defined below, also has a model geometry of the form described above. 
  See Figure~\ref{NaO} for an illustration of a generalized $\Theta$--graph. 
  
\begin{defn}
\label{def:RACG}
Let $\Gamma$ be a finite simplicial graph with vertex set $S$ and edge set $E$. The \emph{right-angled Coxeter group associated to $\Gamma$} is the group $W_{\Gamma}$ with presentation
\[
  W_{\Gamma} = \bigpresentation{s \in S}{\text{$s^2 = 1$ for all $s \in S$, and $[s,t]=1$ for all $\{s,t\} \in E$}}.
\]
If $\Gamma$ has no triangles (i.e. $\Gamma$ is a flag simplicial complex) then $\Gamma$ is equal to the \emph{nerve} of $W_\Gamma$ in the sense of Davis \cite{davis}.

A {\it generalized $\Theta$--graph} is a graph that contains two vertices of valance $k$ and with $k$ edges connecting these two vertices. In addition, the $i^{th}$ edge of the graph is subdivided into $n_i+1$ edges by inserting $n_i$ vertices along this edge. The resulting graph is denoted $\Theta(n_1, \ldots, n_k)$, and we always assume $1 \leq n_i \leq n_j$ for $i<j$. The {\it linear degree} $\ell \in \N$ of a generalized $\Theta$--graph $\Theta(n_1, \ldots, n_k)$ is the cardinality of the $n_i$ that equal one, and the {\it hyperbolic degree} is equal to $k-\ell$. Let $\mathcal{W}$ denote the set of all right-angled Coxeter groups with generalized $\Theta$--graph nerve.
\end{defn}

  A group in $\cW$ is $\delta$--hyperbolic if and only if its linear degree is at most one. The quasi-isometry classification of the $\delta$--hyperbolic groups in $\cW$ follows from Dani--Thomas \cite{danithomas}. In particular, there are infinitely many quasi-isometry classes among $\delta$--hyperbolic groups in~$\cW$. We give the quasi-isometry classification of the remaining non-hyperbolic groups in $\cW$; it follows from the next theorem that there are three quasi-isometry classes among the non-hyperbolic groups in $\cW$. 
  
  \begin{thm} \label{sec1_thm_cW_QI}
   Let $\Theta$ and $\Theta'$ be generalized $\Theta$--graphs with linear degree $\ell \geq 2$ and $\ell' \geq 2$, respectively, and hyperbolic degree $h \geq 0$ and $h'\geq 0$, respectively.  Then the non-hyperbolic right-angled Coxeter groups $W_{\Theta}$ and $W_{\Theta'}$ are quasi-isometric if and only if one of the following three conditions holds:
  \begin{enumerate}
   \item $\ell = \ell' = 2$ and $h,h'\geq 1$; or,
   \item $\ell,\ell' \geq 3$ and $h,h'\geq 1$; or,
   \item $\ell,\ell' \geq 3$, and $h=h'=0$.
  \end{enumerate}
  \end{thm}

  By Caprace \cite{caprace,caprace-erra} each group in $\cW$ is hyperbolic relative to the right-angled Coxeter subgroup with nerve the union of the vertices of valance $k$ and the set of paths of length two connecting these two vertices. We apply work of \Drutu--Sapir \cite{drutusapir} on quasi-isometries between relatively hyperbolic groups to distinguish the quasi-isometry classes within $\cW$. 
  
  A consequence of our quasi-isometry classification of the model geometries of groups in $\cC$ and groups in $\cW$ is that each surface amalgam in $\cC$ is quasi-isometric to a right-angled Coxeter group in $\cW$, but the converse does not hold; see Corollary~\ref{cor:main3}.

  \subsubsection{Abstract commensurability classes}
  
  Theorem~\ref{thm:sec1_CW_AC} states that each group in $\cC$ is abstractly commensurable to a group in $\cW$. Each group in $\cC$ is the fundamental group of a complex consisting of two surfaces and an annulus identified using prescribed gluing maps, and each group in $\cW$ is the orbifold fundamental group of an orbicomplex built from right-angled reflection orbifolds. We prove that for each surface complex there is an associated orbicomplex so that these spaces have homotopy equivalent finite-sheeted covering spaces. The hyperbolic case $m=n=1$ was established previously by Stark \cite[Section~5.2]{stark}.

  The abstract commensurability classification is known for $\delta$--hyperbolic groups in $\cC$ and $\cW$ and is open for non-hyperbolic groups in $\cC$ and $\cW$. A consequence of Theorem~\ref{thm:sec1_CW_AC} is that solving the abstract commensurability classification within $\cC$ reduces to the question of solving the abstract commensurability classification problem within $\cW$. The abstract commensurability classification for $\delta$--hyperbolic groups in $\cW$ is given by Dani--Stark--Thomas \cite{danistarkthomas}. In the $\delta$--hyperbolic setting, if two groups in $\cW$ have the same linear degree, then the groups are abstractly commensurable if and only if their Euler characteristic vectors are commensurable vectors (see Definition~\ref{eulercharvector}). In the non-hyperbolic setting, we prove in Lemmas \ref{hypdegree} and~\ref{cover} that the commensurability classification reduces to the case in which two groups in $\cW$ have the same linear degree, and we prove if such groups have commensurable Euler characteristic vectors, then the groups are abstractly commensurable. It is an open problem to determine whether the converse holds as well. 

\subsection{Outline of the paper}
The preliminary metric notions are given in Section~\ref{sec:Preliminaries}.
Section~\ref{sec:ModelSpaces} contains the construction of the model spaces and model geometries for surface amalgams and right-angled Coxeter groups considered in this paper. 
Section~\ref{sec:ACclasses} contains results on the abstract commensurability classes within $\cC$ and $\cW$. In Section~\ref{sec:AC_thm_proof} we prove each group in $\cC$ is virtually a $3$--manifold group and is abstractly commensurable to a group in $\cW$. Section~\ref{sec:3man} contains the classification of which groups in $\cC$ are the fundamental group of a $3$--manifold and the proof that each group in $\cW$ acts properly on a contractible $3$--manifold. 
Appendix~\ref{sec:QIClassification} contains the quasi-isometry classification; the construction of quasi-isometries is given in Section~\ref{sec:QIConstruction}; the quasi-isometry classification within $\cW$ is given in Section~\ref{sec:QIRACG}; and, the quasi-isometry classification within $\cC$ is given in Section~\ref{sec:QISurfaceAmalgam}. 

\subsection*{Acknowledgments}
The authors are thankful for insightful discussions with Kevin Schreve and Genevieve Walsh and for comments from Chris Cashen on a draft of the paper. The authors also benefited from helpful conversations with Ric Ancel, Pallavi Dani, Craig Guilbault, Jason Manning, Boris Okun, Kim Ruane, and Kevin Whyte about ideas related to this paper.  This work was partially supported by a grant from the Simons Foundation ($\#318815$ to G. Christopher Hruska). The second author was partially supported by the Azrieli Foundation and ISF grant 1941/14. 

\section{Preliminaries}
\label{sec:Preliminaries}

\begin{defn}
Let $(X,d_X)$ and $(Y,d_Y)$ be metric spaces. A map $\Phi$ from $X$ to $Y$ is an \emph{$(L,C)$--quasi-isometry} if there are constants $L\geq 1$ and $C \geq 0$ such that the following hold:
\begin{enumerate}
\item The map $\Phi$ is an {\it $(L,C)$--quasi-isometric embedding}: for all $x_1, x_2 \in X$,
\[
   \frac{1}{L}\,d_X(x_1,x_2)-C\leq d_Y\bigl(\Phi(x_1),\Phi(x_2)\bigr)\leq L\,d_X(x_1,x_2)+C.
\]
\item The map $\Phi$ is {\it $C$--quasi-surjective}: every point of $Y$ lies in the $C$--neighborhood of $f(X)$. 
\end{enumerate}
\end{defn}

\begin{defn}
 Let $(X,d_X)$ and $(Y,d_Y)$ be metric spaces. A map $\Phi$ from $X$ to $Y$ is {\it $K$--bilipschitz} if there exists $K \geq 1$ so that for all $x_1, x_2 \in X$,
\[
   \frac{1}{K}\,d_X(x_1,x_2) \leq d_Y\bigl(\Phi(x_1), \Phi(x_2)\bigr) \leq K\,d_Y(x_1, x_2).
\] 
\end{defn}

\begin{defn}
Metric spaces $X$ and $Y$ are \emph{quasi-isometric} if there is a quasi-isometry from $X$ to $Y$. Two finitely generated groups are {\it quasi-isometric} if their Cayley graphs constructed with respect to finite generating sets are quasi-isometric.
\end{defn}

\begin{defn}
 A {\it model geometry} for a group $G$ is a proper metric space on which $G$ acts properly discontinuously and cocompactly by isometries. 
\end{defn}

A group $G$ is quasi-isometric to any model geometry for $G$. The groups considered in this paper are fundamental groups of finite graphs of groups (see \cite{serre,ScottWall79}). These groups have model geometries that are graphs of spaces in the following sense. 

\begin{defn}
\label{union}
Let $Y_1$ and $Y_2$ be topological spaces with subspaces $A_1 \subset Y_1$ and $A_2 \subset Y_2$.  Let $f\colon A_1 \to A_2$ be a homeomorphism.  The space obtained by identifying $Y_1$ and $Y_2$ along $A_1$ and $A_2$ is the space $X = Y_1 \cup_f Y_2$ defined by $Y_1 \sqcup Y_2 /\bigl(y \sim f(y)\bigr)$ for all $y \in A_1$. If $A$ is the image of $A_1$ and $A_2$ under the quotient map, we also use the notation $X = Y_1 \cup_A Y_2$. 
\end{defn}

\begin{defn}
 Let $G = (V,E)$ be a graph. 
 A {\it geometric graph of spaces} $\Gamma$ consists of a set of vertex spaces $\{X_v \, | \, v \in V\}$
 and a set of edge spaces $\{ X_e \, | \, e \in E\}$ so that the vertex and edge spaces are geodesic metric spaces, and there are isometric embeddings $X_e \rightarrow X_v$ and $X_e \rightarrow X_w$ as convex subsets for each edge $e = \{v,w\} \in E$.  The {\it geometric realization of $\Gamma$} is the metric space $X$ consisting of the disjoint union of the vertex and edge spaces, identified according to the adjacencies of $G$, and given the induced path metric. Observe that all edge and vertex spaces include as convex subspaces of $X$. The {\it underlying graph} of the graph of spaces $\Gamma$ is the abstract graph $G$ specifying $\Gamma$. When the underlying graph of $\Gamma$ is a tree, $\Gamma$ is a \emph{tree of spaces}. 
\end{defn}

The graphs of spaces defined in this paper differ slightly from a common definition of a graph of spaces. Often, one takes the product of each edge space $E$ with an interval $[0,1]$ and glues the spaces $E \times \{0\}$ and $E \times \{1\}$ to the incident vertex spaces; in our setting, the edge spaces are directly glued to vertex spaces. 

\section{Model spaces}
\label{sec:ModelSpaces}

\subsection{Construction of model spaces}

The model spaces considered in this paper are geometric trees of spaces that contain the following as vertex spaces. 

\begin{defn}[Fattened tree, \cite{behrstockneumann}] Let $T$ be a tree whose vertices have valence in the interval $[3,K]$ for some $K$. Fix a positive constant $L$ and assume that $T$ has been given a simplicial metric in which each edge has length between $1$ and $L$. The {\it fattened tree $X$} has each edge $E$ replaced by a strip isometric to $E \times [-\epsilon, \epsilon]$ for some $\epsilon >0$, and each vertex of valence $k$ replaced by a regular $k$--gon with side lengths $2\epsilon$ and so that around the boundary of the polygon the strips replacing the incoming edges of the vertex are attached in some given order. Let $X_0$ be similarly constructed, but starting with the regular $3$--valence tree with all edges having length one and with $\epsilon = \frac{1}{2}$; call $X_0$ the {\it standard fattened tree}.  
\label{fattenedtree}
\end{defn}

\begin{defn} \label{treedef}
For each positive integer $n$, the \emph{regular tree} $T_n$ is the tree in which every vertex has valence $n$. We define $T_0$ to be a single point. For positive integers $m$ and $n$ the \emph{biregular tree} $T_{m,n}$ is the tree in which every vertex has valence $m$ or $n$ so that if two vertices of $T_{m,n}$ are adjacent, one of them has valence $m$ and the other has valence $n$. Metrize each tree so that each edge has length one. 
\end{defn}

\begin{defn} Let $T$ be a tree and $u$ a vertex of $T$. In the metric space $T\times \RR$ each line $\{u\}\times \RR$ is an \emph{essential line}. 
\end{defn}

In Construction~\ref{racgmodel} and Construction~\ref{surfmodel}, we show that each group considered in this paper has a model geometry of the following form.

\begin{cons}[Model spaces] \label{modelspaces}
 Let $m,n,s \geq 1$ be integers, and let $\mathbb{X} = \{X_1, \ldots, X_t\}$ be a finite set of fattened trees. We say $Y$ is {\it a model space of type $(m,n, \mathbb{X},s)$} if $Y$ is the geometric realization of a geometric tree of spaces $\Gamma$ where
  \begin{enumerate}
   \item The underlying tree of $\Gamma$ is bipartite with vertex spaces of two types: either a copy of a fattened tree $X_i \in \XX$ or a copy of $T_{m,n}\times \RR$. Each edge space is a bi-infinite line. 
   
   \item Each essential line in each copy of $T_{m,n} \times \R$ is identified by an isometry to one boundary component of $s$ fattened trees, each isometric to some $X_i\in \XX$.
   
   \item Each boundary component of each fattened tree is identified to exactly one copy of $T_{m,n} \times \R$. 
  \end{enumerate}
 If $\XX = \{X_0\}$, where $X_0$ is the standard fattened tree, we call $Y$ the {\it standard model space} of type $(m,n,s)$. Note that if $Y$ is of type $(m,n,0)$, then $Y$ is isometric to $T_{m,n} \times \R$. 
\end{cons}

\subsection{Right-angled Coxeter groups}

\begin{defn}[Generalized $\Theta$--graph] 
Let $k\geq 3$ and $1 \leq n_1 \leq  n_2 \leq \cdots \leq n_k$ be positive integers. Let $\Psi_k$ be the graph with
two vertices $a$ and $b$ each of valence $k$ and with $k$ edges $e_1, e_2, \cdots e_k$ connecting $a$ and $b$. The \emph{generalized $\Theta$--graph} $\Theta=\Theta(n_1, n_2, \cdots n_k)$ is obtained by subdividing the edge $e_i$ of $\Psi_k$ into $n_i+1$ edges by inserting $n_i$ new vertices along $e_i$ for $1 \leq i \leq k$. An example appears in Figure~\ref{NaO}.

The vertices $a$ and $b$ are called the {\it essential vertices} of $\Theta$, and each path obtained by subdividing the edge $e_i$ is called an {\it essential path of degree $n_i$}. The {\it linear part} of $\Theta$, denoted $\Theta_L$, is the subgraph of $\Theta$ that consists of the union of all essential paths of degree $1$. The {\it hyperbolic part} of $\Theta$, denoted $\Theta_H$, is the subgraph of $\Theta$ consisting of all essential paths of degree at least two. The number of essential paths in $\Theta_L$ is called the {\it linear degree} of $\Theta$, and the number of essential paths in $\Theta_H$ is called the {\it hyperbolic degree} of $\Theta$. 
\end{defn}

\begin{figure}
\centering
\includegraphics[scale=0.60]{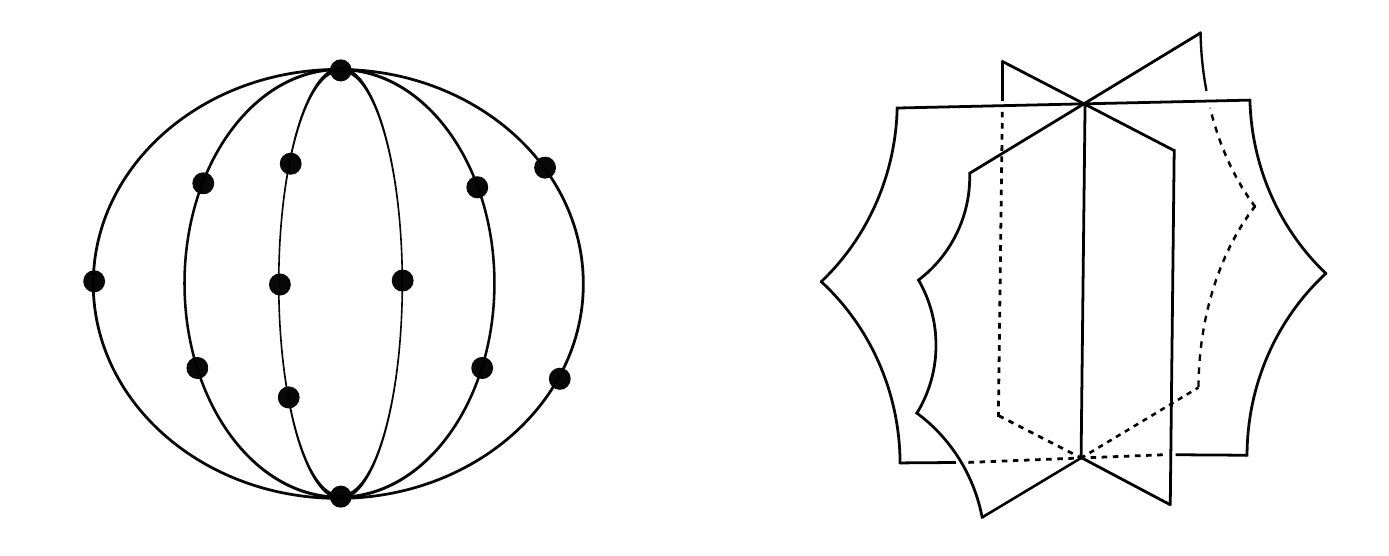}
\caption{{\small On the left is the generalized $\Theta$--graph $\Theta=\Theta(1, 1, 2, 2, 2, 3)$. On the right is an orbi-complex $\mathcal{O}_{\Theta}$ with orbifold fundamental group the right-angled Coxeter group with defining graph $\Theta$. All edges of the orbi-complex are reflection edges except for the branching edge. }}
\label{NaO}
\end{figure}

\begin{rem}
A generalized $\Theta$--graph $\Theta$ is the union of its linear part $\Theta_L$ and its hyperbolic part $\Theta_H$. The intersection of these parts is the set of two essential vertices in $\Theta$.
\end{rem}

\subsubsection{Model geometry}

\begin{defn}[Davis complex]
Given a nontrivial, connected, finite, simplicial, triangle-free graph $\Gamma$ with a set $S$ of vertices, the \emph{Davis complex} $\Sigma_{\Gamma}$ is the Cayley $2$--complex for the presentation of the right-angled Coxeter group $W_{\Gamma}$ given above, in which each disk bounded by a loop with label $s^2$ for $s$ in $S$ has been collapsed to an unoriented edge with label $s$. Then the $1$--skeleton of $\Sigma_{\Gamma}$ is the Cayley graph of $W_{\Gamma}$ with respect to the generating set $S$. Since all relators in this presentation other than $s^2 = 1$ are of the form $stst = 1$, the space $\Sigma_{\Gamma}$ is a square complex. Moreover, the Davis complex $\Sigma_{\Gamma}$ is a $\CAT(0)$ space, and the group $W_{\Gamma}$ acts properly and cocompactly on $\Sigma_{\Gamma}$ (see \cite{davis}).
\end{defn}

\begin{cons}[Model geometry for right-angled Coxeter groups with generalized $\Theta$--graph nerve]
\label{racgmodel}
Let $\Theta = \Theta(n_1, \ldots, n_k)$ be a generalized $\Theta$--graph, and let $W_\Theta$ be the right-angled Coxeter group with defining graph $\Theta$. The quotient of the Davis complex $\Sigma_\Theta$ is a reflection orbi-complex $\mathcal{O}_\Theta$ with orbifold fundamental group $W_\Theta$. The space $\mathcal{O}_\Theta$ is the union of right-angled reflection orbifolds each with one non-reflection edge so that the orbifolds are identified to each other along their non-reflection edges to form one branching edge. An illustration of $\mathcal{O}_\Theta$ appears in Figure~\ref{NaO}. We metrize $\mathcal{O}_\Theta$ so that the universal cover of the orbi-complex, which is topologically the Davis complex, is isometric to a model space given in Construction~\ref{modelspaces}.  

We first metrize the orbifolds in $\mathcal{O}_\Theta$ coming from the hyperbolic part of $\Theta$ so that the universal cover of each orbifold is a fattened tree. Suppose $n_1 = \cdots = n_{\ell} = 1$, so the linear degree of $\Theta$ is $\ell$ with $0 \leq \ell \leq k$, and the hyperbolic degree of $\Theta$ is $k-\ell$. For each $i > \ell$, let $O_i \subset \mathcal{O}_\Theta$ be a orbifold with one non-reflection edge and $n_i +2 \geq 4$ reflection edges labeled in order $\sigma_1, \ldots, \sigma_{n_i+2}$ so that $\sigma_i$ and $\sigma_{i+1}$ are adjacent for $1 \leq i \leq 2n_i+1$. Subdivide $\sigma_2$ and $\sigma_{n_i+1}$ into two edges by inserting a vertex in the middle of each edge. Subdivide $\sigma_3, \ldots, \sigma_n$ into three edges by inserting two equally spaced vertices on each edge. Subdivide the boundary component of $O_i$ into $n_i$ edges by inserting $n_i -1$ equally spaced vertices along the boundary edge. Connect the vertices along the boundary edge to the new vertices along the reflection edges so that edge vertex on the boundary line has two out-going edges, the edges do not intersect, and, together with the boundary of $O_i$, they subdivide $O_i$ into a cell complex where each cell has four sides, as illustrated in Figure~\ref{racgcell}. Metrize these cells as Euclidean rectangles so that the new interior edges each have length one, each pair of edges emanating from the same vertex on the boundary edge span a square of side lengths one, and each edge along the boundary has length $L_i \geq 1$, so that $n_i L_i = n_j L_j = L$ for $\ell < i,j \leq k$. Since $O_i$ deformation retracts onto its reflection edges, whose universal cover is a tree, the universal cover of $O_i$ with this metric is a fattened tree: the vertex spaces are squares and the edge spaces have length either $L_i$ or $2L_i$. 
 
 \begin{figure}
\centering
\includegraphics[scale=0.15]{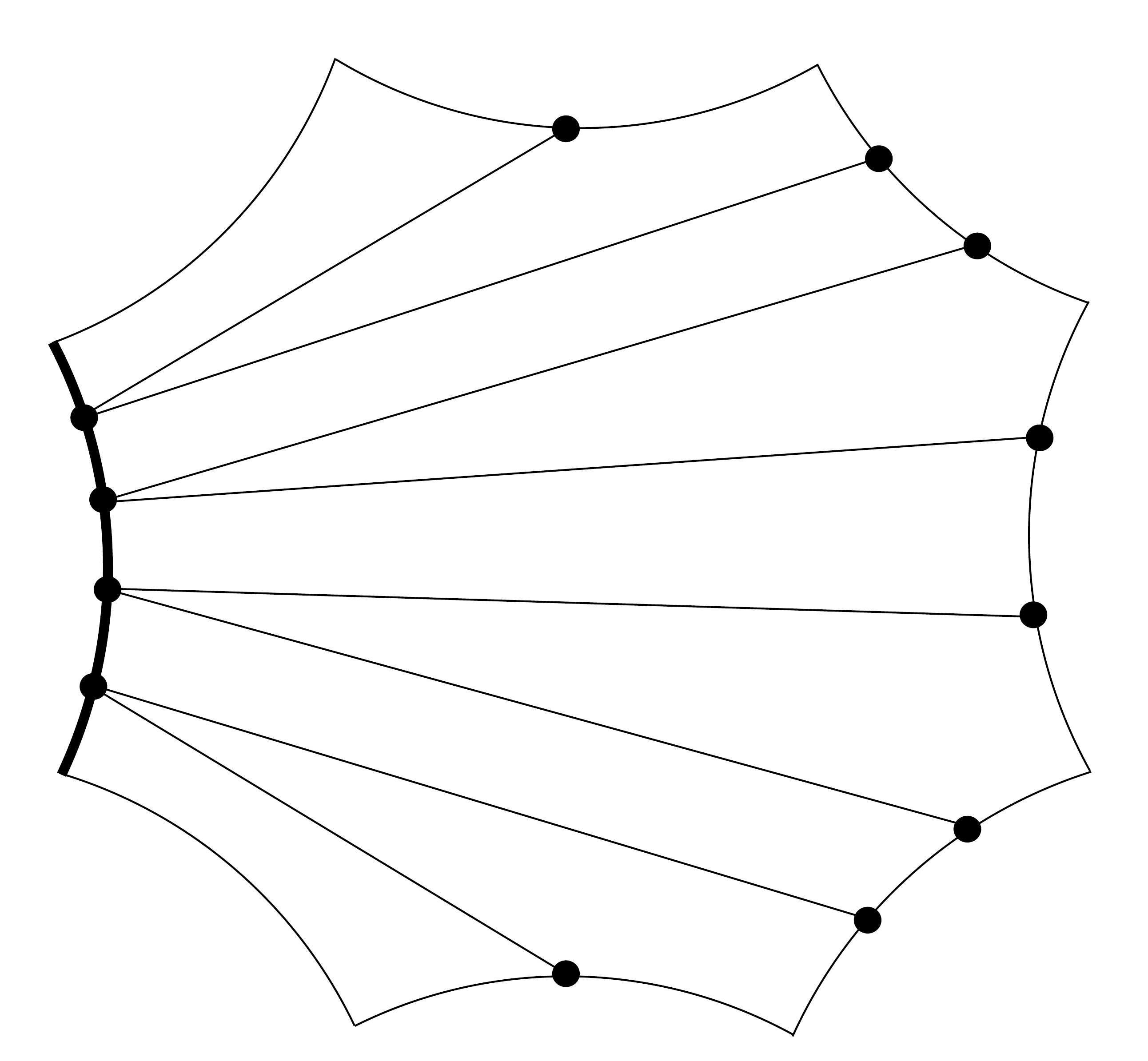}
\caption{{\small A right-angled reflection orbifold with one darkened non-reflection edge and seven reflection edges. The orbifold has been given a metric as a rectangular complex, as described in Construction~\ref{racgmodel}, so that the universal cover of the orbifold is a fattened tree. The orbifold deformation retracts onto a neighborhood of the union of the reflection edges, whose universal cover is a tree.   }}
\label{racgcell}
\end{figure}
 
 Metrize each linear orbifold in $\mathcal{O}_\Theta$ as a rectangle with non-reflection side of length $L$ and adjacent pair of edges of length one. Then, the orbifolds coming from the essential paths in $\Theta$ can be glued by an isometry along their non-reflection edges to form $\mathcal{O}_\Theta$, whose universal cover is isometric to a model space of type $(\ell, \ell, \XX,k-\ell)$ given in Construction~\ref{modelspaces}.
\end{cons}

\subsection{Surface group amalgams} \label{subsec:QI_surf_amals}

\subsubsection{Model geometry}

 \begin{notation}
 \label{hatcher}
  Let $A_{m,n}$ be homeomorphic to the quotient space of the cylinder $S^1 \times [0,1]$ under the identifications $(z,0) \sim (e^{2\pi i/m}z, 0)$ and $(z,1) \sim (e^{2\pi i/n}z, 1)$. 
  Let $A_m$ and $A_n$ be the two halves of $A_{m,n}$ formed by the quotients of $S^1 \times \bigl[0, \frac{1}{2}\bigr]$ and $S^1 \times \bigl[\frac{1}{2}, 1\bigr]$. Then, $A_m$ and $A_n$ are mapping cylinders of $z \mapsto z^m$ and $z \mapsto z^n$, respectively. 
  The fundamental group of $A_{m,n}$ is $\presentation{a,b}{a^m = b^n}$.
 \end{notation}

    \begin{figure}
      \begin{overpic}[scale=.5, tics=5]{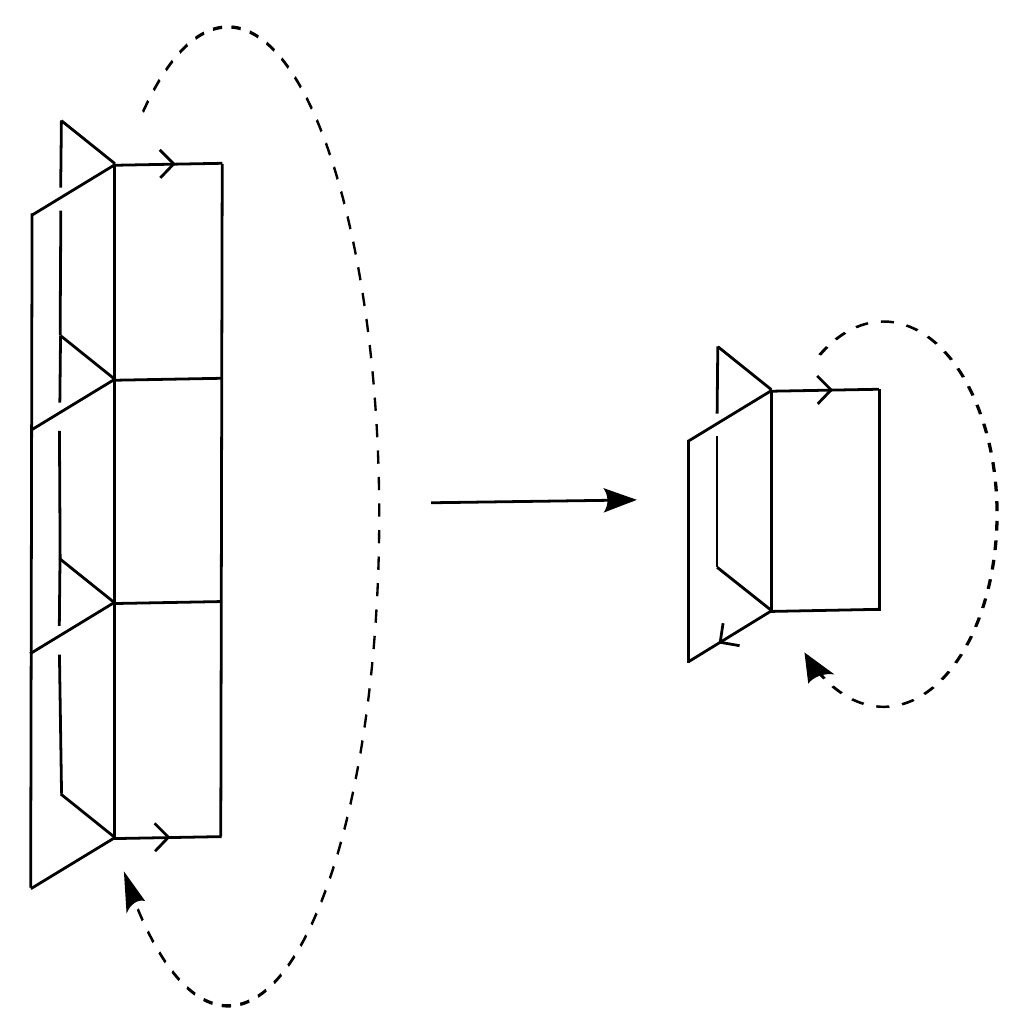}
      \put(50,54){$3$}
      \put(75,20){$\frac{2\pi}{3}$}
      \put(69,12){\Small{rotation}}
      \end{overpic}
	\caption{{\small  The degree--$3$ cover $Y_3 \times S^1 \rightarrow A_3$ described in Lemma~\ref{kmncover}. The cover is given by the $\Z/3\Z$ action on the space on the left generated by the screw motion that cyclically permutes the three fins and translates upward by one unit.  }}
      \label{figure:A_n_cover}
     \end{figure}
 
 As explained by Hatcher in \cite[Example~1.24]{hatcher}, the $2$--complex $A_{m,n}$ is a $2$--dimensional spine of a torus knot complement when $\gcd(m,n)=1$ and is closely related to its standard Seifert fibered structure.
The following result is analogous to the well-known fact that a Seifert fibered space has a finite-sheeted cover that is a product of a surface with a circle.  The result is implicit in the discussion of Example~1.35 from \cite{hatcher}.

 \begin{lem} \label{kmncover}
  The space $K_{m,n} \times S^1$ forms a degree--$mn$ cover of the space $A_{m,n}$, where $K_{m,n}$ denotes the complete bipartite graph on $m$ and $n$ vertices, and $A_{m,n}$ is described above. 
 \end{lem}

  \begin{proof}
   Let $Y_n$ denote the finite biregular tree $T_{1,n}$, the ``$n$--pointed asterisk'' (see Definition~\ref{treedef}). Let $\rho_n\colon Y_n \rightarrow Y_n$ denote a cyclic permutation of the $n$ leaves of $Y_n$. We claim that $Y_n \times S^1$ forms a degree--$n$ cover of the space $A_n$ defined in Notation~\ref{hatcher}. Indeed, let $\tau_n\colon S^1 \rightarrow S^1$ be given by $z \mapsto e^{2\pi i/n}z$. Then, there is a $\Z /n\Z$ action on $Y_n \times S^1$ generated by the map $(\rho_n, \tau_n)\colon Y_n \times S^1 \rightarrow Y_n \times S^1$. The quotient under this action is the space $A_n$. An example is given in Figure~\ref{figure:A_n_cover}. Each boundary curve of $Y_n \times S^1$ is mapped homeomorphically to the boundary curve of $A_n$. Similarly, the space $Y_m \times S^1$ forms a degree--$m$ cover of $A_m$ so that each boundary curve of $Y_m \times S^1$ is mapped homeomorphically to the boundary curve of $A_n$. 

   The covering maps $Y_n \times S^1 \rightarrow A_n$ and $Y_m \times S^1 \rightarrow A_m$ restrict to homeomorphisms on the boundary curves, so copies of these spaces may be glued together to form a cover of $A_{m,n}$ homeomorphic to $K_{m,n} \times S^1$ as illustrated in the central column of Figure~\ref{bounding_pair_cover}.
Indeed $K_{m,n} \times S^1$ consists of $m$ copies of $Y_n\times S^1$ and $n$ copies of $Y_m \times S^1$ such that each pair of opposite types intersects in a single circle.
We define a covering map from $K_{m,n} \times S^1 \to A_{m,n}$ in the obvious way by pasting together the various covering maps defined on these subsets.
  \end{proof}

The following corollary is immediate.  See \cite[Example~1.35]{hatcher} for an alternate explanation.

\begin{cor}
\label{A_mn_universal}
 The universal cover of $A_{m,n}$ is $T_{m,n} \times \R$, where $T_{m,n}$ denotes the biregular tree with vertices of valance $m$ and $n$. \qed
\end{cor}

\begin{cons}[Model geometry for $G \in \mathcal{C}$] 
\label{surfmodel}

Let $X$ be the union of surfaces $S_g$, $S_h$, and an annulus, so that one boundary component of the annulus is attached to the curve $\gamma_g$ on $S_g$ by a degree--$m$ map and the other boundary component of the annulus is attached to the curve $\gamma_h$ on $S_h$ by a degree--$n$ map. Then, $A_{m,n}$ is a subspace of $X$.

We will metrize $X$ so that $\widetilde{X}$, the universal cover of $X$, is isometric to a model space given in Construction~\ref{modelspaces}. We first put a metric on the surfaces with boundary $S_g \backslash \gamma_g$ and $S_h \backslash \gamma_h$ so that the universal cover is a fattened tree. The construction is illustrated in Figure~\ref{surfmetric}.

\begin{figure}
\centering
\includegraphics[scale=.4]{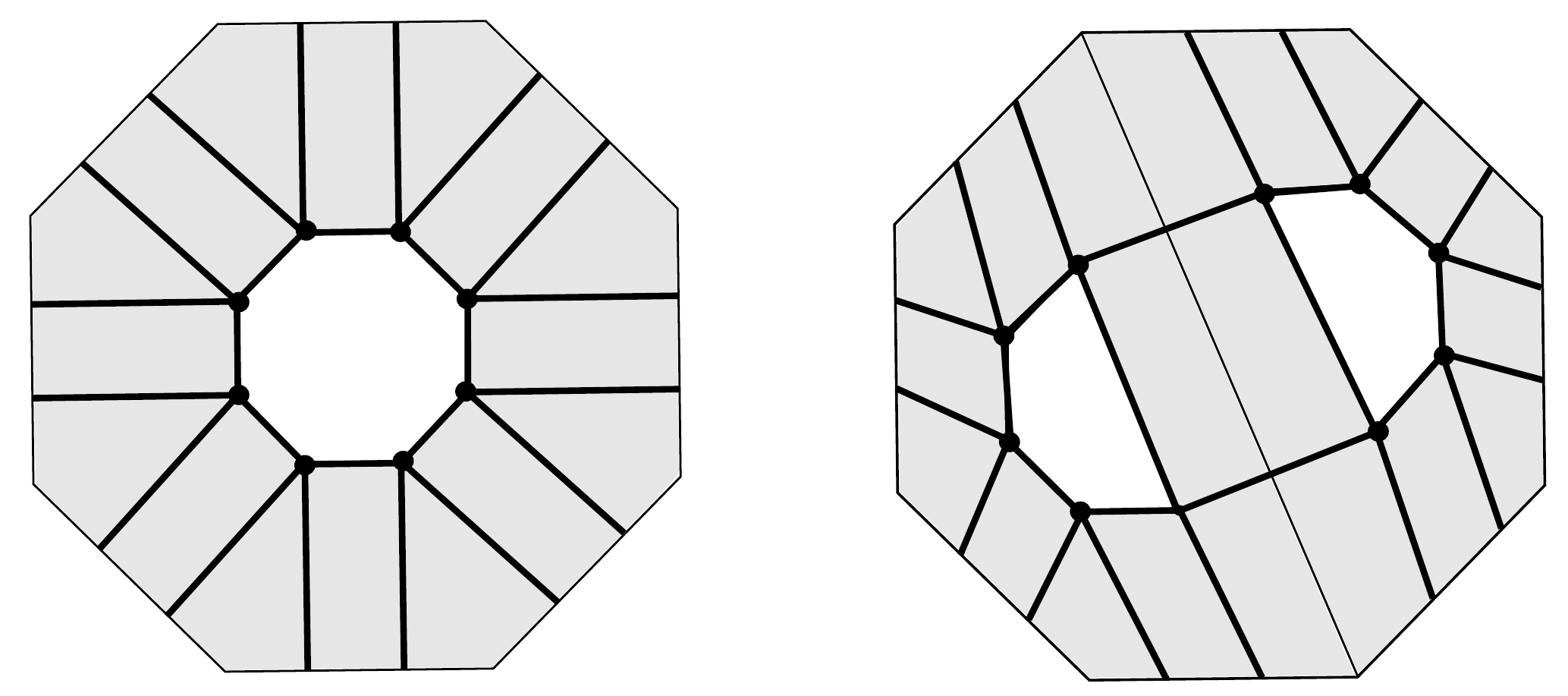}
\caption{{\small Identify opposite sides of the outer octagons to form a surface with genus two; on the left the surface has one boundary component, and on the right, the surface has two boundary components. The universal cover of the surface, with the metric on the cells described in Construction~\ref{surfmodel}, is a fattened tree.   }}
\label{surfmetric}
\end{figure}

Let $S$ be a connected surface with genus $g \geq 1$ and either one or two boundary components. Realize $S$ as a regular $4g$--gon $P$ with opposite sides identified and with boundary components in the interior of the polygon. 

Suppose first that $S$ has one boundary component. Subdivide each side of $P$ into three segments of equal length, and subdivide the boundary component of $S$ into $4g$ segments of equal length. On each side of $P$ connect each of the two interior vertices to a vertex on the boundary component of $S$ by edges whose interiors are pairwise disjoint such that each vertex on the boundary of $S$ connects to two vertices that lie on adjacent sides of the polygon $P$. This construction realizes $S$ as a cell complex whose $0$--cells all lie on the boundary of $S$. The $2$--cells of $S$ that do not contain corners of $P$ are rectangles, and $S$ contains one more $2$--cell that is a $4g$--gon containing all of the corners of the polygon $P$. Metrize the cell complex $S$ so that the $4g$--gon is a regular Euclidean polygon with sides of length one, each rectangle is isometric to $[-\frac{1}{2},\frac{1}{2}] \times [0,L]$, where the value of $L \geq 1$ depends on $G$ and is chosen below, and the edges $[-\frac{1}{2},\frac{1}{2}] \times \{0\}$ and $[-\frac{1}{2},\frac{1}{2}] \times \{L\}$ of the rectangle are identified to the edges of the $4g$--gon by isometries. Then, the boundary curve of $S$ is a circle of length $4gL$, and, since $S$ deformation retracts onto the sides of $P$, the universal cover of $S$ is isometric to a fattened tree. 

Suppose now that $S$ has two boundary components. 
Add a diagonal to $P$ that subdivides $P$ into two $(2g+1)$--gons each containing one boundary component of $S$.
Repeat the above subdivision construction on each of the two polygons to produce a cell structure on $S$ as illustrated on the right side of Figure~\ref{surfmetric} whose $0$--cells all lie on the boundary of $S$. The $2$--cells of $S$ consist of rectangles that do not contain corners of $P$ along with a single $(4g+2)$--gon that contains all of the corners of the polygon $P$.  Endow $S$ with a piecewise Euclidean metric as above such that each rectangle is isometric to $\bigl[ -\frac{1}{2},\frac{1}{2} \bigr] \times [0,K]$ for a constant $K$ chosen below.
Then each boundary curve of $S$ is a circle of length $(2g+1)K$, and, since $S$ deformation retracts onto the graph formed by the sides and chosen diagonal of $P$, the universal cover of $S$ is isometric to a fattened tree. 

The choice of the length of the rectangles in the construction depends on the group $G$. If $\gamma_g$ separates $S_g$ into two subsurfaces of genus $g_1$ and $g_2$, the lengths $L_1, L_2 \geq 1$ must satisfy $4g_1 L_1 = 4g_2 L_2$, which will equal the length of $\gamma_g$. If $\gamma_g$ is non-separating, its length is equal to $(2g_3+1)L_3$, where $g_3$ is the genus of $S_g \backslash \gamma_g$ and $L_3$ is the length of the rectangle. The length of $\gamma_h$ is similar. In addition, choose the constants $L_i$ and $K_i$ so that $m$ times the length of $\gamma_g$ is equal to $n$ times the length of $\gamma_h$, which is equal to some constant $C$.  

By Corollary~\ref{A_mn_universal}, the universal cover of the subspace $A_{m,n}$ is $T_{m,n} \times \R$. So the universal cover $\widetilde{X}$ of $X$ is isometric to a model space of type $(m,n,\XX,2)$ given in Construction~\ref{modelspaces}. 
\end{cons}

\begin{rem}
 The construction above can be used to metrize any surface $S$ with nonempty boundary so that the universal cover is a fattened tree: realize $S$ as a polygon with sides identified and boundary components in the interior of the polygon. Choose an embedded graph that $S$ deformation retracts onto and perform the construction described above. 
\end{rem}

\section{Abstract commensurability classes}
\label{sec:ACclasses}

In this section, we prove Theorem~\ref{3mancovers}, which states that each surface amalgam in $\cC$ is abstractly commensurable to a right-angled Coxeter group in $\cW$. In Section~\ref{sec:Euler}, we define the Euler characteristic vector for a group in $\cW$; preliminary covering maps are given in Section~\ref{sec:prelim_covers}; Theorem~\ref{3mancovers} is proven in Section~\ref{sec:AC_thm_proof}; and, Section~\ref{sec:add_comm} contains additional commensurabilities. 

\subsection{Euler characteristic vector}
\label{sec:Euler}

Our results about abstract commensurability classes involve the Euler characteristic vector. The abstract commensurability classification is open for the non-hyperbolic groups in $\mathcal{C}$ and $\mathcal{W}$; in the $\delta$--hyperbolic setting, the abstract commensurability classes can be given in terms of the Euler characteristic vector; see \cite{danistarkthomas}. 

 \begin{defn} (Euler characteristic vector.)
 Let $\Theta = \Theta(n_1, \ldots, n_k)$ be a generalized $\Theta$-graph, and let $W_{\Theta}$ be the associated right-angled Coxeter group.
 As described in Construction~\ref{racgmodel}, the group $W_{\Theta}$ is the orbifold fundamental group of a right-angled reflection orbi-complex $\mathcal{O}_{\Theta}$. The space $\mathcal{O}_{\Theta}$ is the union of right-angled reflection orbifolds $P_i$ for $1 \leq i \leq k$, where $P_i$ has $n_i+2$ reflection edges and one non-reflection edge, and the union identifies the non-reflection edges of $P_i$ for $1 \leq i \leq k$.  The orbifold Euler characteristic of $P_i$ is given by the formula
\[
   \chi_i = \chi(P_i) = 1 - \left(1+\frac{n_i+2}{2}\right) + \left(\frac{2}{2} + \frac{n_i+1}{4}\right) = \frac{1-n_i}{4}.
\]
If $\ell$ is the linear degree of $\Theta$ and $h$ is the hyperbolic degree of $\Theta$, then $\chi_i = 0$ for $1 \leq i \leq \ell$ and $\chi_i <0$ for $\ell+1 \leq i \leq \ell+h$. The {\it Euler characteristic vector} of $W_{\Theta}$ is 
\[
   \chi_{\Theta} = (\chi_1, \ldots, \chi_k) = (\underbrace{0, \ldots, 0}_{\text{$\ell$}}, \chi_{\ell+1}, \ldots, \chi_{\ell+h} ).
\] 
\label{eulercharvector}
\end{defn}

\subsection{Covering maps and finite-index subgroups}
\label{sec:prelim_covers}

In this section, we describe some tools used to construct finite-index subgroups. First, we use the following lemma from \cite{danistarkthomas}, which proves that an orbi-complex with fundamental group a right-angled Coxeter group with generalized $\Theta$--graph nerve has a finite cover consisting of surfaces with two boundary curves, identified to each other along these curves. An example of the covering space appears in the lower left of Figure~\ref{2foldcovers}.

 \begin{lem}[\cite{danistarkthomas}, Section~4]
 \label{degree16}
      Let $W$ be a right-angled Coxeter group with generalized $\Theta$--graph nerve and with Euler characteristic vector $(\chi_1,\ldots, \chi_k)$. Let $\mathcal{O}$ be the orbi-complex given in Construction~\ref{racgmodel} with orbifold fundamental group $W$. Then there is a degree--$16$ cover $Z \rightarrow \mathcal{O}$, so that $Z$ consists of $k$ surfaces $S_1, \ldots, S_k$, where $S_i$ has two boundary curves, $C_{i1}$ and $C_{i2}$ and $\chi(S_i) = 16\chi_i$ for $1 \leq i \leq k$. The curves $\set{C_{i1}}{1 \leq i \leq k}$ are identified to form a single curve $C_1$ and the curves $\set{C_{i2}}{1 \leq i \leq k}$ are identified to form a single curve $C_2$.
 \end{lem}
 
The following lemma states that a $d$--fold covering of the boundary of a surface can be extended to a $d$--fold covering of the entire surface provided a natural parity condition holds. 

 \begin{lem}[\cite{neumann}, Lemma~3.2]
 \label{neumann}
  If $S$ is an orientable surface of positive genus, a degree $d \geq 1$ is specified, and for each boundary component of $S$ a collection of degrees summing to $d$ is also specified, then a connected $d$--fold covering $S'$ of $S$ exists with prescribed degrees on the boundary components of $S'$ over each boundary component of $S$ if and only if the prescribed number of boundary components of the cover has the same parity as $d\chi(S)$. 
 \end{lem}

 We make repeated use of the following simple procedure to paste together covering maps. 
 
 \begin{lem} \label{coverpaste}
  Let $A$ and $B$ be topological spaces with $A \cup B = X$ and $A \cap B = C \neq \emptyset$. Let $p_1\colon\widehat{A} \rightarrow A$ and $p_2\colon\widehat{B} \rightarrow B$ be covering maps, and let $C_1 = p_1^{-1}(C) \subset \widehat{A}$ and $C_2 = p_2^{-1}(C)\subset \widehat{B}$. Suppose $f\colon C_1 \rightarrow C_2$ is a covering isomorphism, so $p_1= p_2 \circ f$ on $C_1$. Then, $\widehat{X} = \widehat{A} \cup_f \widehat{B}$, (with the union notation defined in Definition~\ref{union}) is a cover of $X$. \end{lem}

 \begin{defn} \label{bounding}
  A \emph{bounding pair} in a surface is a pair of disjoint, homologous, nonseparating simple closed curves. 
 \end{defn}
 
 An image of bounding pairs and an illustration of the following lemma appears in Figure~\ref{bounding_pair_cover}.

 \begin{lem} \label{boundingpair}
  If $\gamma\colon S^1 \rightarrow S_g$ is an essential simple closed curve, then there exists a double cover $\widehat{S}_g \rightarrow S_g$ where the pre-image of $\gamma$ is a bounding pair. 
 \end{lem}
 
 \begin{proof}
  Let $\gamma\colon S^1 \rightarrow S_g$ be an essential simple closed curve. Suppose first that $\gamma$ is non-separating. Then, the surface $S = S_g \backslash \gamma$ is connected and has boundary $S^1 \sqcup S^1$. The surface $\widehat{S}_g = S \cup_{S^1 \sqcup S^1} S$ obtained by identifying two copies of $S$ together along their boundary forms a degree--$2$ cover of $S_g$ with the covering map given by rotation about the handle formed by the gluing. The pre-image of $\gamma$ in $\widehat{S}_g$ is a bounding pair that separates $\widehat{S}_g$ into two subsurfaces each of Euler characteristic $\chi(S_g)$. 
  
  Suppose now that $\gamma$ separates $S_g$ into two subsurfaces $A$ and $B$. By Lemma~\ref{neumann}, there is a degree--$2$ cover $\widehat{A} \rightarrow A$ so that $\widehat{A}$ has two boundary components, which each cover $\gamma$ by degree one. Similarly, by Lemma~\ref{neumann}, there is a degree--$2$ cover $\widehat{B} \rightarrow B$ so that $\widehat{B}$ has two boundary components, which each cover $\gamma$ by degree one. Let $\widehat{S}_g = \widehat{A} \cup_{S^1 \sqcup S^1} \widehat{B}$. By Lemma~\ref{coverpaste}, $\widehat{S}_g$ forms a degree--$2$ cover of $S_g$. The pre-image of $\gamma$ in $\widehat{S}_g$ is a bounding pair that separates $\widehat{S}_g$ into two surfaces of Euler characteristic $2\chi(A)$ and $2\chi(B)$. 
 \end{proof}

\subsection{Surface group amalgams are virtual $3$--manifold groups and are abstractly commensurable to right-angled Coxeter groups}
\label{sec:AC_thm_proof}

  Each surface amalgam in $\cC$ is the fundamental group of a complex $X$ defined in Construction~\ref{surfmodel}, and each right-angled Coxeter group in $\cW$ is the orbifold fundamental group of an orbicomplex~$\cO$ defined in Construction~\ref{racgmodel}. To prove that each surface amalgam in $\cC$ is abstractly commensurable to a right-angled Coxeter group in $\cW$, for each complex $X$ we define an orbicomplex $\cO$, and
  we construct covers of $X$ and $\cO$ which are homotopy equivalent. Along the way, in Theorem~\ref{thm:surf_amal_3mancover}, we construct a cover of $X$ which is a deformation retract of a $3$--manifold with boundary, which can be constructed using similar arguments to those in \cite[Section~8]{kapovichkleiner}. The remaining covering maps and proof of commensurability are given in Theorem~\ref{3mancovers}. An alternative proof of Theorem~\ref{thm:surf_amal_3mancover} proved using Coxeter groups is given in Corollary~\ref{cor:surf_3man}.

  \begin{thm} \label{thm:surf_amal_3mancover}
   If $G \in \cC_{m,n}$, then $G$ has a finite-index subgroup that acts freely on a $3$--manifold.
  \end{thm}
  \begin{proof}
   Let $G \cong \pi_1(S_g) *_{\la a^m = b^n \ra} \pi_1(S_h) \in \mathcal{C}_{m,n}$. Let $X$ be the space given in Construction~\ref{surfmodel} with fundamental group $G$. We construct two finite covers $X_2 \xrightarrow{mn} X_1 \xrightarrow{2} X$ so that the space $X_2$ is a deformation retract of a $3$--manifold with boundary.
   
   Recall that the space $X$ consists of the surfaces $S_g$ and $S_h$, and suppose $a_0\colon S^1 \rightarrow S_g$ and $b_0\colon S^1 \rightarrow S_h$ are essential simple closed curves in the homotopy classes of $a \in \pi_1(S_g)$ and $b \in \pi_1(S_h)$, respectively. The space $X$ also contains the complex $A_{m,n} = S^1 \times I / \sim$ defined in Notation~\ref{hatcher}; the quotient of $S^1 \times \{0\}$ under the identification is glued to the curve $a_0$, and the quotient of $S^1 \times \{1\}$ under the identification is glued to the curve $b_0$. 
 
 We first describe the cover $X_1 \xrightarrow{2} X$; an illustration of the covering map appears in Figure~\ref{bounding_pair_cover}. The space $X_1$ contains a double cover of $S_g$, a double cover of $S_h$, and a double cover of $A_{m,n}$. By Lemma~\ref{boundingpair}, there is a double cover $\widehat{S}_g \rightarrow S_g$ so that the pre-image of $a_0$ in $\widehat{S}_g$ is a bounding pair (see Definition~\ref{bounding}) that separates $\widehat{S}_g$ into subsurfaces $S_1$ and $S_2$ of Euler characteristics $v_1$ and $v_2$. Similarly, by Lemma~\ref{boundingpair}, there is a double cover $\widehat{S}_h \rightarrow S_h$ so that the pre-image of $b_0$ in $\widehat{S}_h$ is a bounding pair that separates $\widehat{S}_h$ into subsurfaces $S_3$ and $S_4$ of Euler characteristics $v_3$ and $v_4$. Let $\widehat{A}_{m,n}$ be a degree--$2$ cover of $A_{m,n}$ that consists of two disjoint copies of $A_{m,n}$. Then the pre-image of $a_0$ in $\widehat{S}_g$ and the pre-image of $a_0$ in $\widehat{A}_{m,n}$ is $S^1 \sqcup S^1$ with covering maps given by homeomorphisms. Likewise, the pre-image of $b_0$ in $\widehat{S}_h$ and the pre-image of $b_0$ in $\widehat{A}_{m,n}$ is $S^1 \sqcup S^1$ with covering maps given by homeomorphisms. Thus, by Lemma~\ref{coverpaste}, the pre-image of $a_0$ in $\widehat{S}_g$ and in $\widehat{A}_{m,n}$ can be identified by a homeomorphism and the pre-image of $b_0$ in $\widehat{S}_h$ and in $\widehat{A}_{m,n}$ can be identified by a homeomorphism to form the space $X_1$ which forms a degree-$2$ cover of $X$.  
 
     \begin{figure}
      \begin{overpic}[scale=.5,  tics=5]{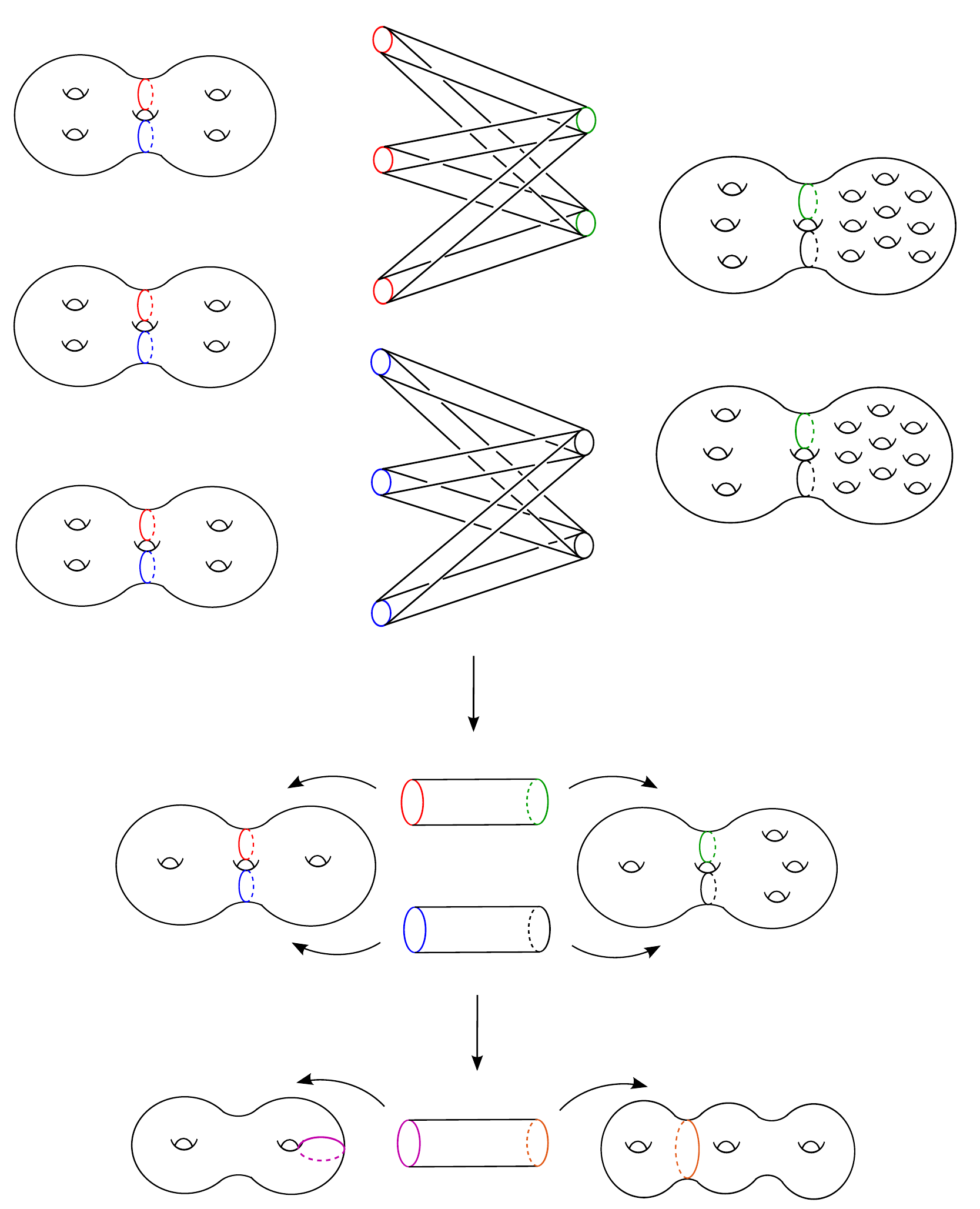}
      \put(1,5){$X$}
      \put(1,30){$X_1$}
      \put(65,95){$X_2$}
      \put(41,14){\Small{$2$}}
      \put(41,42){\Small{$6$}}
      \put(27,11){\Small{$2$}}
      \put(50,11){\Small{$3$}}
      \put(27,18){\Small{$2$}}
      \put(50,18){\Small{$3$}}
      \put(27,37){\Small{$2$}}
      \put(50,37){\Small{$3$}}
      \put(24.5,6.5){\small{$a_0$}}
      \put(56,-1){\small{$b_0$}}
      \put(28,96){\small{$c_1$}}
      \put(28,86){\small{$c_2$}}
      \put(28,76){\small{$c_3$}}
      \put(27.5,69){\small{$d_1$}}
      \put(27.5,59){\small{$d_2$}}
      \put(27.5,49){\small{$d_3$}}
      \put(50,89){\small{$e_1$}}
      \put(50,80){\small{$e_2$}}
      \put(50,63){\small{$f_1$}}
      \put(50,54){\small{$f_2$}}
      \put(13.5,92){\small{$c_1$}}
      \put(13.5,74.5){\small{$c_2$}}
      \put(13.5,56){\small{$c_3$}}
      \put(13.5,88){\small{$d_1$}}
      \put(13.5,70.5){\small{$d_2$}}
      \put(13.5,52){\small{$d_3$}}
      \put(63,83){\small{$e_1$}}
      \put(63,64){\small{$e_2$}}
      \put(63,78){\small{$f_1$}}
      \put(63,59.5){\small{$f_2$}}
      \end{overpic}
	\caption{{\small  A degree--$2$ cover $X_1 \rightarrow X$ and a degree--$6$ cover $X_2 \rightarrow X_1$. In $X$ and $X_1$, curves of the same color are glued by a degree--$2$ or degree--$3$ map of the circle as indicated. The curves $a_0$ and $b_0$ in $X$ each have pre-image in $X_1$ a {\it bounding pair}. The space $X_2$ consists of five surfaces and two copies of $K_{2,3} \times S^1$, and the colored curves are identified by a homeomorphism as indicated by the labeling.  }}
      \label{bounding_pair_cover}
     \end{figure}
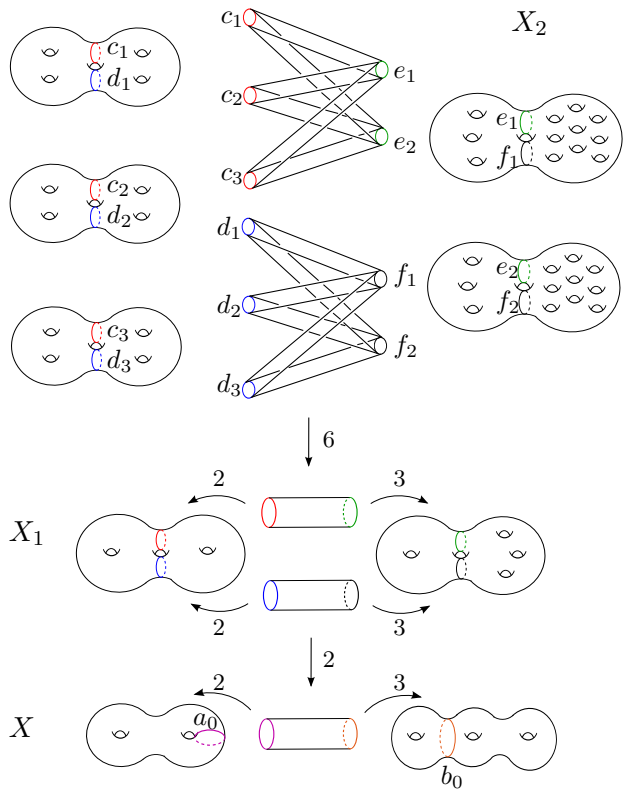

 We next describe the cover $X_2 \xrightarrow{mn} X_1$, illustrated in Figure~\ref{bounding_pair_cover}. The space $X_2$ contains a degree--$mn$ cover of each of the surfaces $S_i$ defined in the last paragraph for $1 \leq i \leq 4$; we describe these covers first. For $i=1,2$ color one boundary component of $S_i$ red and color the other boundary component of $S_i$ blue. By Lemma~\ref{neumann}, for $i=1,2$, there is a degree--$m$ cover $\widetilde{S}_i \rightarrow S_i$ so that $\chi(\widetilde{S}_i) = mv_i$; the surface $\widetilde{S}_i$ contains two boundary curves, one colored red and one colored blue; and, the red boundary curve of $\widetilde{S}_i$ covers the red boundary curve of $S_i$ by degree $m$, and the blue boundary curve of $\widetilde{S}_i$ covers the blue boundary curve of $S_i$ by degree $m$. Similarly, for $i=3,4$, color one boundary curve of $S_i$ green and color the other boundary component of $S_i$ black. By Lemma~\ref{neumann}, for $i=3,4$, there is a degree--$n$ cover $\widetilde{S}_i \rightarrow S_i$ so that $\chi(\widetilde{S}_i) = nv_i$; the surface $\widetilde{S}_i$ contains two boundary curves, one colored green and one colored black; and, the green boundary curve of $\widetilde{S}_i$ covers the green boundary curve of $S_i$ by degree $n$, and the black boundary curve of $\widetilde{S}_i$ covers the black boundary curve of $S_i$ by degree $n$. For $i=1,2$, let $T_i$ be $n$ disjoint copies of $\widetilde{S}_i$, and for $i=3,4$, let $T_i$ be $m$ disjoint copies of $\widetilde{S}_i$.  Then, for $i \in \{1,2,3,4\}$, the (typically disconnected) surface $T_i$ forms a degree--$mn$ cover of $S_i$.  
 
 The space $X_2$ also contains a degree--$mn$ cover of each of the two copies of $A_{m,n}$ in $X_1$, given as follows. Let $p\colon K_{m,n} \times S^1 \rightarrow A_{m,n}$ be the degree--$mn$ cover from Lemma~\ref{kmncover}. If $v_m \in K_{m,n}$ and $v_n \in K_{m,n}$ denote vertices of valance $m$ and $n$, respectively, then $p$ restricts to a degree--$m$ cover on $v_m \times S^1$ over the quotient of $S^1 \times \{0\} \subset A_{m,n}$ and restricts to a degree--$n$ cover on $v_n \times S^1$ over the quotient of $S^1 \times \{1\} \subset A_{m,n}$. Let $K_1$ and $K_2$ denote two copies of $K_{m,n} \times S^1$.
 
 The space $X_2$ is obtained by identifying the spaces $\{K_1, K_2, T_1, \ldots, T_4\}$ in the following way. The space $K_1$ contains $n$ circles of the form $v_m \times S^1$, where $v_m \in K_{m,n}$ is a vertex of valance $m$. This subspace of $K_1$ is homeomorphic to $\bigsqcup_{i=1}^n S^1$. For $i=1,2$, the set of $n$ red circles in $T_i$ is also homeomorphic to $\bigsqcup_{i=1}^n S^1$. Identify these three subspaces homeomorphic to $\bigsqcup_{i=1}^n S^1$ by a homeomorphism. Similarly, the space $K_1$ contains $m$ circles of the form $v_n \times S^1$, where $v_n \in K_{m,n}$ is a vertex of valance $n$. This subspace of $K_1$ is homeomorphic to $\bigsqcup_{i=1}^m S^1$. For $i=3,4$, the set of green curves in $T_i$ is also homeomorphic to $\bigsqcup_{i=1}^m S^1$. Identify these three subspaces homeomorphic to $\bigsqcup_{i=1}^m S^1$ by a homeomorphism. Analogously, the space $K_2$ contains $n$ circles of the form $v_m \times S^1$, where $v_m \in K_{m,n}$ is a vertex of valance $m$. This subspace of $K_2$ is homeomorphic to $\bigsqcup_{i=1}^n S^1$. For $i=1,2$, the set of $n$ blue circles in $T_i$ is also homeomorphic to $\bigsqcup_{i=1}^n S^1$. Identify these three subspaces homeomorphic to $\bigsqcup_{i=1}^n S^1$ by a homeomorphism. Similarly, the space $K_2$ contains $m$ circles of the form $v_n \times S^1$, where $v_n \in K_{m,n}$ is a vertex of valance $n$. This subspace of $K_2$ is homeomorphic to $\bigsqcup_{i=1}^m S^1$. For $i=3,4$, the set of black curves in $T_i$ is also homeomorphic to $\bigsqcup_{i=1}^m S^1$. Identify these three subspaces homeomorphic to $\bigsqcup_{i=1}^m S^1$ by a homeomorphism to form $X_2$. By Lemma~\ref{coverpaste}, $X_2$ forms a degree--$mn$ cover of~$X_1$. 
 
 The space $X_2$ has a structure similar to the graph of spaces described by Kapovich--Kleiner in \cite[Section~8]{kapovichkleiner}. Roughly speaking, they build a $3$--manifold with boundary which deformation retracts to $X_2$ by replacing each branching curve with a solid torus, taking the product of each surface with boundary with an interval, and gluing the boundary annuli of the thickened surfaces to annuli on the solid tori.  
 Thus, the thickening of the space $X_2$ is a deformation retraction of a $3$--manifold with boundary. Therefore, $G$ has a finite-index subgroup which acts freely on a $3$--manifold.  \end{proof}

 \begin{thm} \label{3mancovers}
If $G \in \mathcal{C}_{m,n}$, then $G$ is abstractly commensurable to a right-angled Coxeter group with a generalized $\Theta$--graph nerve. More specifically, suppose $G \cong \pi_1(S_g) *_{\la a^m = b^n \ra} \pi_1(S_h) \in \mathcal{C}_{m,n}$. Then, $G$ is abstractly commensurable to the right-angled Coxeter group $W$ whose nerve is the generalized $\Theta$--graph with Euler characteristic vector
\[
   w = (\!\!\!\!\!\!\!\!\!\underbrace{0, \ldots, 0}_\text{$ 2(mn-m-n+1)$}\!\!\!\!\!\!\!\!\! ,  \underbrace{mv_1, \ldots, mv_1}_\text{\ \ \ \ $n$ times}, \underbrace{mv_2, \ldots, mv_2}_\text{$n$ times}, \underbrace{nv_3, \ldots, nv_3}_\text{$m$ times}, \underbrace{nv_4, \ldots, nv_4}_\text{$m$ times}   ),
\]
 where if $a \in \pi_1(S_g)$ is the homotopy class of the curve $a_0$ and $a_0$ is non-separating, define $v_1 = v_2 = \chi(S_g)$, and if $a_0$ separates $S_g$ into two subsurfaces $A$ and $B$ with $\chi(A) \leq \chi(B)$, define $v_1 = 2\chi(A)$ and $v_2 = 2\chi(B)$. Define $v_3$ and $v_4$ analogously. 
 \end{thm}

 \begin{proof}
 Let $G \cong \pi_1(S_g) *_{\la a^m = b^n \ra} \pi_1(S_h) \in \mathcal{C}_{m,n}$ and let $W$ be the right-angled Coxeter group given in the statement of the theorem. Let $N =mn-m-n+1$. Let $\mathcal{O}$ be the orbi-complex given by Construction~\ref{racgmodel} with orbifold fundamental group $W$, and let $X$ be the space given by Construction~\ref{surfmodel} with fundamental group $G$.
 To prove $G$ and $W$ are abstractly commensurable,
 we construct two finite towers of maps $Z_2 \xrightarrow{2} Z_1 \xrightarrow{16} \mathcal{O}$ and $X_5 \xrightarrow{16} X_4 \xrightarrow{2} X_3 \xrightarrow{\simeq} X_2 \xrightarrow{mn} X_1 \xrightarrow{2} X$.  All spaces in these towers are connected $2$--complexes, the map $X_3 \rightarrow X_2$ is a homotopy equivalence, and each remaining map is a covering map with the degree specified above the arrow.  The $2$--complexes $Z_2$ and $X_5$ at the top of the towers are homeomorphic.  Since each map in the two towers is $\pi_1$--injective with finite index image, we get an isomorphism $\pi_1(Z_2) \cong \pi_1(X_5)$ between finite index subgroups of $G$ and $W$ as desired.
 
 We first describe the covering maps $Z_2 \xrightarrow{2} Z_1 \xrightarrow{16} \mathcal{O}$. Let $Z_1 \xrightarrow{16} \mathcal{O}$ be the cover given by Lemma~\ref{degree16}. An example of $Z_1$ appears on the lower left side Figure~\ref{2foldcovers}. By construction, $Z_1$ has two singular curves $C_1$ and $C_2$. There is a set $\mathcal{A}$ of $2N$ annuli in $Z_1$ and each annulus in $\mathcal{A}$ has one boundary curve glued to $C_1$ and one boundary curve glued to $C_2$. In addition, $Z_1$ has a set $\mathcal{B}$ of $2(m+n)$ surfaces with negative Euler characteristic and two boundary components. Each surface in $\mathcal{B}$ has one boundary curve glued to $C_1$ and one boundary curve glued to $C_2$. In particular, the set of surfaces $\mathcal{B}$ consists of $n$ surfaces with Euler characteristic $16mv_1$,  $n$ surfaces with Euler characteristic $16mv_2$, $m$ surfaces with Euler characteristic $16nv_3$, and  $m$ surfaces with Euler characteristic $16nv_4$. Let $Z_2$ be the following space; an example of $Z_2$ appears at the top of Figure~\ref{2foldcovers}. The space $Z_2$ contains four singular curves $D_1, D_2, D_3, D_4$, two copies of the set $\mathcal{A}$ denoted $A$ and $A'$, and two copies of the set $\mathcal{B}$ denoted $B$ and $B'$. Attach to the curves  $\{D_1, D_2\}$ each annulus in $A$ so that each annulus has one boundary curve glued to $D_1$ and the other boundary curve glued to $D_2$. Attach to the curves  $\{D_3, D_4\}$ each annulus in $A'$ so that each annulus has one boundary curve glued to $D_3$ and the other boundary curve glued to $D_4$. Similarly, attach to the curves  $\{D_1, D_4\}$ each surface in the set $B$ so that each surface has one boundary curve glued to $D_1$ and one boundary curve glued to $D_4$. Finally, attach to the curves  $\{D_2, D_3\}$ each surface in the set $B'$ so that each surface has one boundary curve glued to $D_2$ and one boundary curve glued to $D_3$. Then, $Z_2$ forms a degree-$2$ cover of $Z_1$ with the covering map given by rotation as pictured in Figure~\ref{2foldcovers}.

     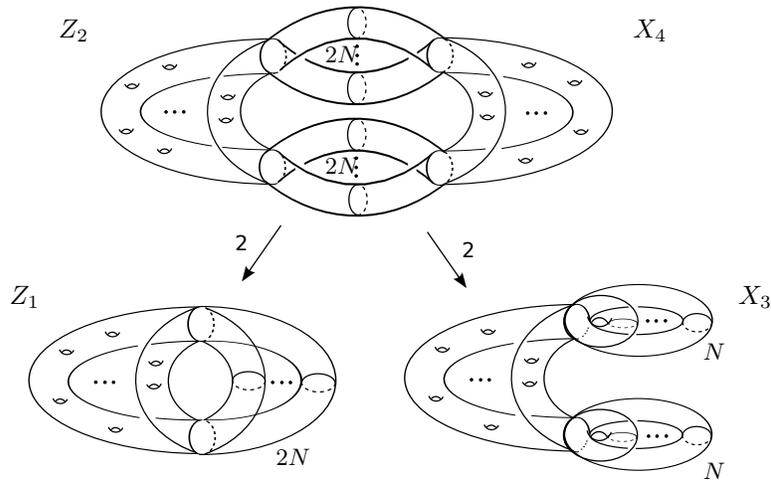
\begin{figure}
      \begin{overpic}[scale=.6,  tics=5]{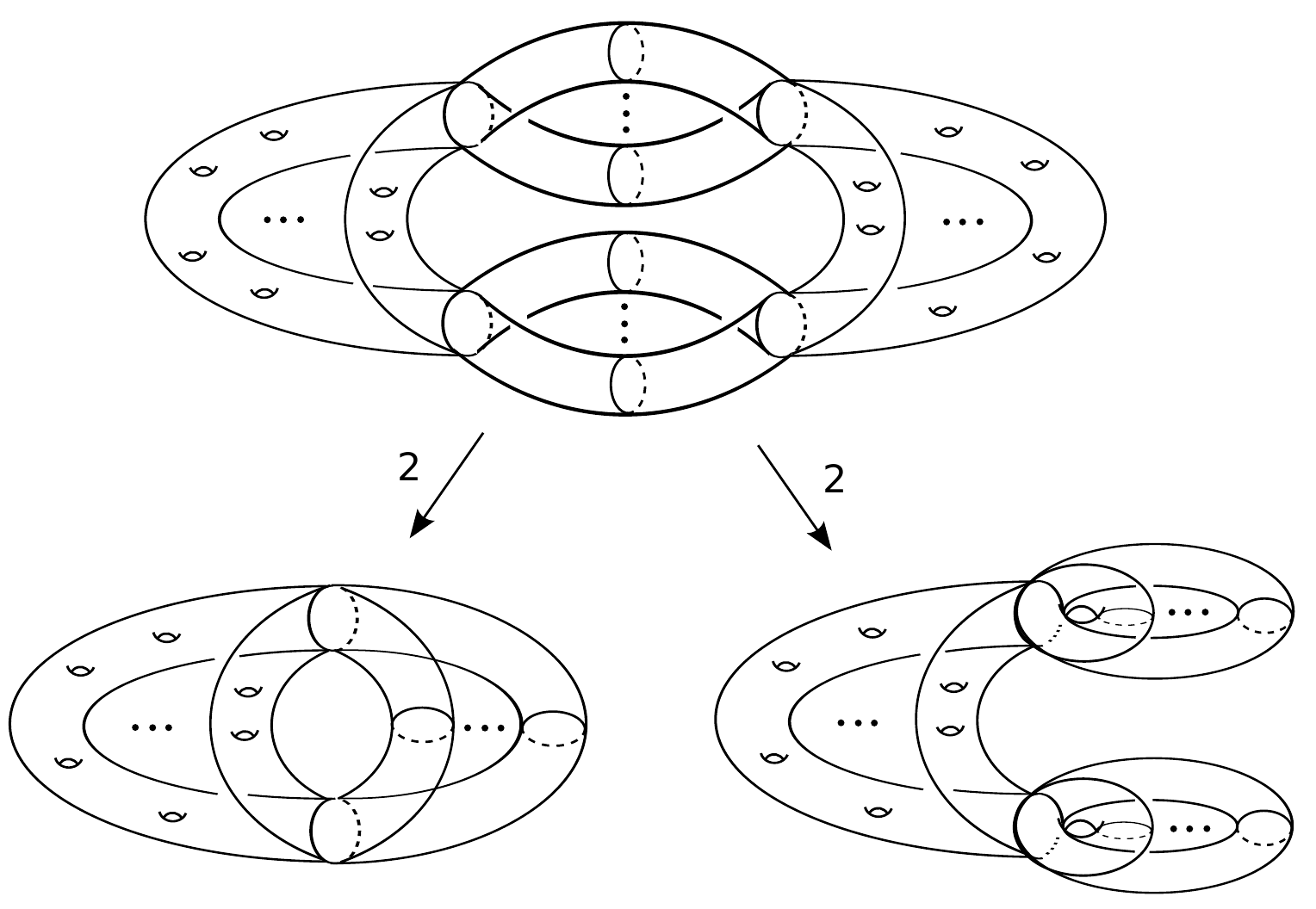}
        \put(43,59.4){\Small{$2N$}}
        \put(43,43.6){\Small{$2N$}}
        \put(36,2){\Small{$2N$}}
        \put(97,17){\Small{$N$}}
        \put(97,0){\Small{$N$}}
        \put(102,25){\small{$X_3$}}
        \put(87,63){\small{$X_4$}}
        \put(-2,25){\small{$Z_1$}}
        \put(5,63){\small{$Z_2$}}
      \end{overpic}
	\caption{{\small Illustrated above are two degree--$2$ covers given by rotation about an axis positioned through the space above. On the top and on the bottom-left, there are collections of $2N$ annuli connecting singular curves; on the bottom-right there are two collections of $N$ tori identified to the singular curves. }}
      \label{2foldcovers}
     \end{figure}
 
 We now describe the maps $X_5 \xrightarrow{16} X_4 \xrightarrow{2} X_3 \xrightarrow{\simeq} X_2 \xrightarrow{mn} X_1 \xrightarrow{2} X$. The space $X$ and the covers $X_2 \xrightarrow{mn} X_1 \xrightarrow{2} X$ are described above in the proof of Theorem~\ref{thm:surf_amal_3mancover}. 
 
 The homotopy equivalence $X_3 \rightarrow X_2$ is given as follows.  The construction is based on the well-known fact that any quotient of a CW-complex formed by collapsing a contractible subcomplex to a point is a homotopy equivalence \cite[Proposition~0.17]{hatcher}.
 In particular the quotient $q\colon K_{m,n} \to K_{m,n} / T_0$ is a homotopy equivalence, where $T_0$ is a maximal subtree of $K_{m,n}$.  The quotient is homeomorphic to the rose $R_N$ on $N$ petals, where $N = mn-m-n+1$.
 It follows that the quotient map $\bar{q}\colon K_{m,n} \times S^1 \rightarrow R_N \times S^1$ collapsing each subspace $T_0 \times \{x\}$ to a point is a homotopy equivalence.
 Let $X_3$ be the quotient space $X_2 / {\sim}$ in which the identification described above is performed separately on each of the two disjoint copies of $K_{m,n} \times S^1$ in $X_2$.
 The quotient map $X_2 \rightarrow X_3$ is again a homotopy equivalence by the homotopy extension property.  Indeed the homotopy extension argument is nearly identical to the proof of  \cite[Proposition~0.17]{hatcher}.  The desired map $X_3 \rightarrow X_2$ is a homotopy inverse of $X_2 \rightarrow X_3$.
 
 Let us take a moment to describe the structure of the quotient space $X_3$ produced above and illustrated in the lower right portion of Figure~\ref{2foldcovers}.  The space $X_3$  contains two singular curves $C$ and $C'$: all red and green curves in $X_2$ are identified, and all blue and black curves in $X_2$ are identified. In $X_3$ there are $N$ tori identified to $C$ along a simple closed curve on each torus, similarly, there are $N$ tori identified to $C'$ along a simple closed curve on each torus. In addition, there is a set $\mathcal{T}$ of $2(m+n)$ surfaces with negative Euler characteristic and two boundary components identified to $C$ and $C'$. In the set $\mathcal{T}$, there are $n$ surfaces with Euler characteristic $mv_1$, $n$ surfaces with Euler characteristic $mv_2$, $m$ surfaces with Euler characteristic $nv_3$, and $m$ surfaces with Euler characteristic $mv_4$.

 There is a degree--$2$ cover $X_4 \rightarrow X_3$ given as follows; an example of the space $X_4$ is pictured on the top of Figure~\ref{2foldcovers}, with the covering map given on the right-hand side. The space $X_4$ contains four singular curves $E_1, E_2, E_3, E_4$, sets $\mathcal{N}$ and $\mathcal{N}'$, which are each homeomorphic to $2N$ annuli, and, sets $T$ and $T'$, which are each homeomorphic to $\mathcal{T}$ described in the previous paragraph. The gluings are given as follows. Attach to the curves $\{E_1, E_2\}$ each annulus in $\mathcal{N}$ so that each annulus has one boundary component glued to the curve $E_1$ and one boundary component glued to the curve $E_2$. Attach to the curves $\{E_3, E_4\}$ each annulus in $\mathcal{N}'$ so that each annulus has one boundary component glued to $E_3$ and one boundary component glued to $E_4$. Similarly, attach to the curves $\{E_1, E_4\}$ each surface in the set $T$ so that each surface has one boundary component glued to $E_1$ and one boundary component glued to $E_4$. Finally, attach to the curves $\{E_2, E_3\}$ each surface in the set $T'$ so that each surface has one boundary component glued to $E_2$ and one boundary component glued to $E_3$. Then $X_4$ forms a degree--$2$ cover of $X_3$ with the covering map given by rotation about an axis that passes vertically through the two sets of annuli shown in Figure~\ref{2foldcovers} and has $N$ annuli from each set on each side. Alternatively, this covering map may be seen by cutting each torus in $X_3$ along a meridian curve parallel to the singular curve, taking two copies of the resulting space, and re-gluing the boundary components in pairs.

 The space $X_4$ has the same underlying structure as the space $Z_2$, but the Euler characteristic of the surfaces in $X_4$ are less than those in $Z_2$ by a factor of $16$; so, construct one final cover $X_5 \rightarrow X_4$ of degree $16$ so that $X_5 \cong Z_2$. By Lemma~\ref{neumann}, there is a set of surfaces $\widetilde{\mathcal{T}}$ that consists of a degree--$16$ cover $\widetilde{S}$ of each $S \in \mathcal{T}$ so that $\widetilde{S}$ has two boundary components and $\chi(\widetilde{S}) = 16mv_i$ for $i=1,2$ or $\chi(\widetilde{S}) = 16nv_j$ for $j=3,4$. Let $X_5$ be the space with the same underlying structure as $X_4$ and formed as follows. The space $X_5$ contains four singular curves $E_1', E_2', E_3', E_4'$, the sets of annuli $\mathcal{N}$ and $\mathcal{N}'$, and spaces $\widetilde{T}$ and $\widetilde{T}'$, each homeomorphic to $\widetilde{\mathcal{T}}$.  The gluings are given as follows. Attach to the curves $\{E_1', E_2'\}$ each annulus in $\mathcal{N}$ so that each annulus has one boundary component glued to the curve $E_1'$ and one boundary component glued to the curve $E_2'$. Attach to the curves $\{E_3', E_4'\}$ each annulus in $\mathcal{N}'$ so that each annulus has one boundary component glued to $E_3'$ and one boundary component glued to $E_4'$. Similarly, attach to the curves $\{E_1', E_4'\}$ each surface in the set $\widetilde{T}$ so that each surface has one boundary component glued to $E_1'$ and one boundary component glued to $E_4'$. Finally, attach to the curves $\{E_2', E_3'\}$ each surface in the set $\widetilde{T}'$ so that each surface has one boundary component glued to $E_2'$ and one boundary component glued to $E_3'$. Then, each surface with boundary in the subspaces $\mathcal{N}$, $\mathcal{N}'$, $\widetilde{T}$, and $\widetilde{T}'$ of $X_5$ covers a corresponding surface in $X_4$ so that the degree restricted to the boundary components is equal to 16. Thus, by Lemma~\ref{coverpaste}, the space $X_5$ forms a degree--$16$ cover of $X_4$. By construction, $X_5$ and $Z_2$ are homeomorphic. Therefore, $G$ and $W$ are abstractly commensurable. 
 \end{proof}

\subsection{Additional commensurabilities} \label{sec:add_comm}

Two vectors $v,w \in \Z^n$ are {\it commensurable} if there exist non-zero integers $K$ and $L$ so that $Kv=Lw$. The following lemma, which generalizes a technique of Crisp--Paoluzzi \cite{crisppaoluzzi} and is illustrated in the left of Figure~\ref{orbicovers}, implies that if two right-angled Coxeter groups in $\cW$ have commensurable Euler characteristic vectors, then they are abstractly commensurable.

\begin{figure}
\centering
\includegraphics[scale=.5]{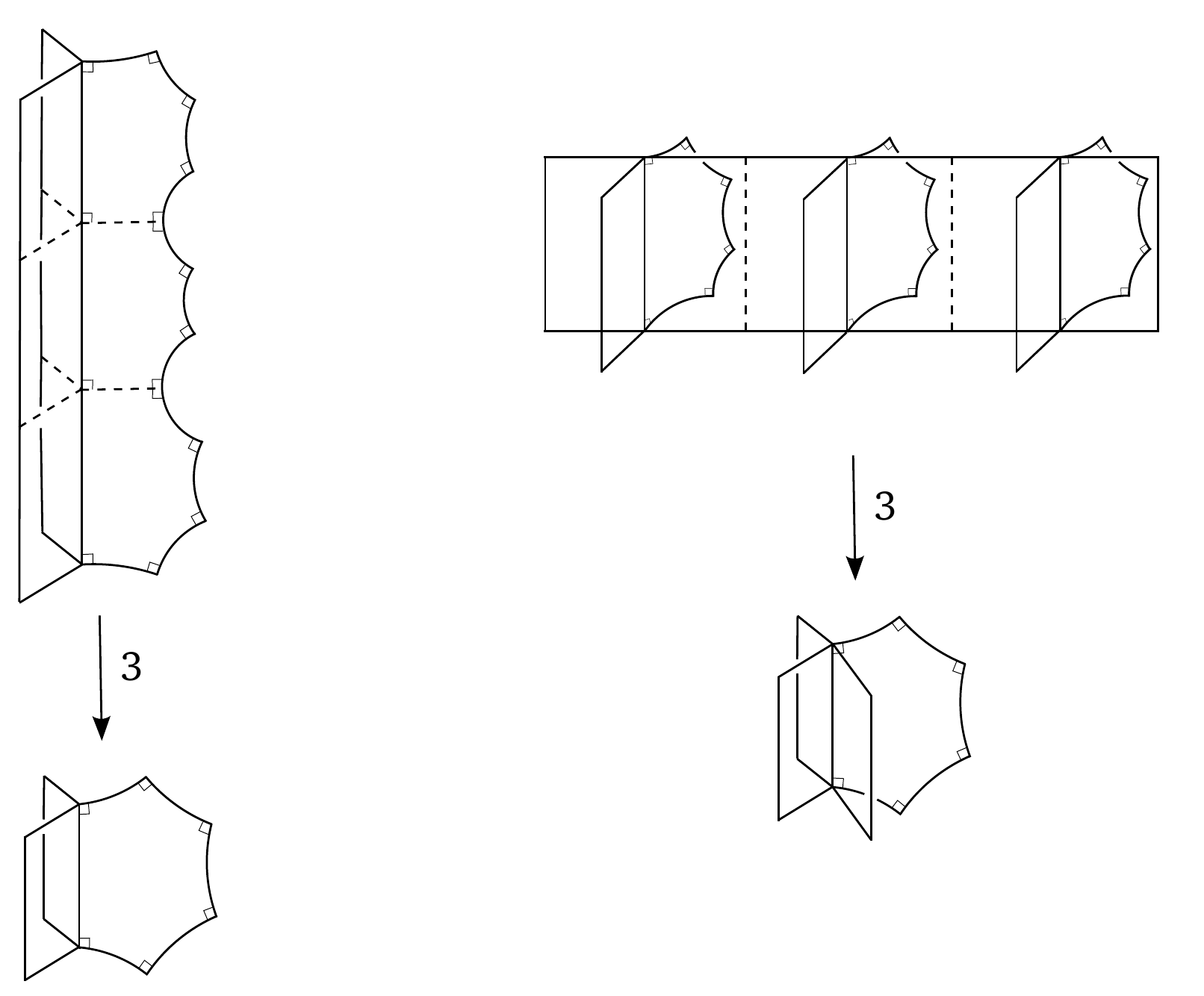}
\caption{{\small Illustrated above are two degree--$3$ covers of an orbi-complex with fundamental group a right-angled Coxeter group. The covers are given locally by reflection along the dashed lines. The orbifold fundamental group of each covering space is a right-angled Coxeter group with generalized $\Theta$--graph nerve.}}
\label{orbicovers}
\end{figure}

\begin{lem} \label{cover}
 Let $W_{\Theta}$ be the right-angled Coxeter group whose nerve is the generalized $\Theta$--graph with Euler characteristic vector
 \[
   w = (\underbrace{0,\ldots, 0}_\text{$\ell$}, \chi_{\ell+1}, \ldots, \chi_{\ell+h}  ).
\]
Then $W_{\Theta}$ is abstractly commensurable with the right-angled Coxeter group $W$ whose nerve is the generalized $\Theta$--graph with Euler characteristic vector $Kw$ for each integer $K>0$.
\end{lem}

\begin{proof}
 Let $\mathcal{O}_{\Theta}$ be the orbi-complex described in Construction~\ref{racgmodel} with orbifold fundamental group $W_{\Theta}$. We show that $\mathcal{O}_{\Theta}$ has a $K$--fold cover $\mathcal{O}$ with orbifold fundamental group $W$. To construct the cover, take $K$ copies of the orbi-complex $\mathcal{O}_{\Theta}$. Glue branching reflection edges incident to the a vertex in one copy to branching reflection edges incident to a vertex in another copy so that orbifolds with the same number of sides are glued together and so that there is one branching edge in the cover $\mathcal{O}$. The orbifold fundamental group $W$ of $\mathcal{O}$ is a right-angled Coxeter group with generalized $\Theta$--graph nerve and linear degree $\ell$. Since Euler characteristic is multiplicative under covering maps, $W$ has Euler characteristic vector $Kw$ as desired.
\end{proof}

\begin{cor}
 If groups $W, W' \in \cW$ have commensurable Euler characteristic vectors, then $W$ and $W'$ are abstractly commensurable.  
\end{cor}

The following lemma, illustrated in the right of Figure~\ref{orbicovers}, states that within each commensurability class in the non-hyperbolic setting, for a given group $W \in \cW$, there are right-angled Coxeter groups in $\cW$ commensurable to $W$ and with arbitrarily large hyperbolic degree. 

\begin{lem} \label{hypdegree}
Let $\Theta$ be the generalized $\Theta$--graph with Euler characteristic vector
\[
   (\underbrace{0,\ldots, 0}_\text{$\ell$}, \chi_{\ell+1}, \ldots, \chi_{\ell+h}  ).
\]
Then for all $m \geq 1$,  the group $W_{\Theta}$ is abstractly commensurable to the right-angled Coxeter group $W$ whose nerve is the generalized $\Theta$--graph with Euler characteristic vector
\[
   (\underbrace{0,\ldots, 0,}_\text{$m(\ell-2)+2$}
   \underbrace{\underbrace{\chi_{\ell+1}, \ldots, \chi_{\ell+h}}, \underbrace{ \chi_{\ell+1}, \ldots, \chi_{\ell+h}}, \ldots, \underbrace{\chi_{\ell+1}, \ldots, \chi_{\ell+h}}}_\text{$m$ times}).
\]
\end{lem}

\begin{proof}
 Let $\mathcal{O}_{\Theta}$ be the orbi-complex described in Construction~\ref{racgmodel} with orbifold fundamental group $W_{\Theta}$. We show that $\mathcal{O}_{\Theta}$ has an $m$-fold cover $\mathcal{O}$ with orbifold fundamental group $W$. To construct the cover, take $m$ copies of the orbi-complex $\mathcal{O}_{\Theta}$ and glue exterior reflection edges of the rectangles to each other in pairs to form an orbi-complex $\mathcal{O}$ with $m$ branching edges and so that the branching edges lie along one rectangular orbifold as pictured in Figure~\ref{orbicovers}. The orbifold covering map is locally trivial away from the glued reflection edges, where the covering map is locally a reflection. The orbi-complex $\mathcal{O}$ has $m$ copies of each hyperbolic reflection orbifold in $\mathcal{O}_{\Theta}$ and $2(\ell-1)+(m-2)(\ell-2) = m(\ell-2)+2$ rectangular orbifolds with one boundary edge attached to the branching lines along their boundary edges. The orbi-complex $\mathcal{O}$ is homotopy equivalent to an orbi-complex with one branching line: collapse the rectangular orbifolds with two boundary lines that are attached by their boundary lines to distinct branching lines. Then, the fundamental group of this orbi-complex is the right-angled Coxeter group whose nerve is the generalized $\Theta$--graph with the desired Euler characteristic vector.
\end{proof}

\section{$3$--manifold groups}
\label{sec:3man}

  In Section~\ref{sec:coarse_separation} we state the results of Kapovich--Kleiner \cite{kapovichkleiner05} that we will need; in Section~\ref{subsec:3mantop} we review some background on $3$--manifolds; and, in Section~\ref{3manproof} we characterize which groups in $\cC$ are fundamental groups of $3$--manifolds, and we show each group in $\cW$ acts properly and cocompactly on a contractible $3$--manifold.

 \subsection{Coarse separation}\label{sec:coarse_separation}
 
 \begin{defn}
  A {\it metric simplicial complex} $X$ is the geometric realization of a connected simplicial complex, metrized so that each edge has length one. The complex $X$ is said to have {\it bounded geometry} if all links have a uniformly bounded number of simplices. A metric simplicial complex is {\it uniformly acyclic} if for every $R_1\in \R$ there exists $R_2\in \R$ such that for each subcomplex $K \subset X$ of diameter less than $R_1$ the inclusion $K \rightarrow N_{R_2}(K)$ induces zero on reduced homology groups. A component $C$ of $X \backslash N_R(K)$ is called {\it deep} if it is not contained within a finite neighborhood of $K$. A subcomplex~$K$ {\it coarsely separates} $X$ if there is an $R$ so that $X \backslash N_R(K)$ has at least two deep components. A map $f\colon X \rightarrow Y$ between metric spaces is {\it coarse Lipschitz} if there exist $L\geq 1$ and $A \geq 0$ so that $d\bigl( f(x), f(x') \bigr) \leq L\,d(x,x') +A$ for all $x,x' \in X$. A coarse Lipschitz map is {\it uniformly proper} if there is a proper function $\phi\colon\R^+ \rightarrow \R^+$ so that $d\bigl(f(x), f(x')\bigr) \geq \phi\bigl(d(x,x')\bigr)$ for all $x,x' \in X$. (In particular, a quasi-isometric embedding is uniformly proper.)
 \end{defn}

  We refer the reader to \cite{kapovichkleiner05} for the definition of a coarse $PD(n)$ space. We will not use the definition explicitly, rather, we will use a characterization of certain coarse $PD(n)$ spaces. 
  
    \begin{lem}[\cite{kapovichkleiner05}, Lemma~6.2]
    \label{lem:char_pd3}
    Let $M$ be an aspherical closed $n$--manifold equipped with a finite triangulation.
    Then its universal cover $\tilde{M}$ is a coarse $PD(n)$ space on which $\pi_1(M)$ acts properly, cocompactly, and simplicially.
  \end{lem}
 
 \begin{lem}[\cite{kapovichkleiner05}, Lemma~7.11]
 \label{separation}
  Let $W$ be a bounded geometry metric simplicial complex which is homeomorphic to a union $W = \bigcup_{i\in I} W_i$ of $k$ half-spaces $W_i \cong \R_{+}^2$ along their boundaries. Assume that for $i \neq j$, the union $W_i \cup W_j$ is uniformly acyclic and is uniformly properly embedded in $W$. Let $f\colon W \rightarrow X$ be a uniformly proper map of $W$ into a coarse $PD(3)$ space $X$. Then $f(W)$ coarsely separates $X$ into $k$ components. Moreover, there is a unique cyclic ordering on the index set $I$ so that for $R$ sufficiently large, the boundary of each deep component $C$ of $X \backslash N_R\bigl(f(W)\bigr)$ is at finite Hausdorff distance from $f(W_i) \cup f(W_j)$, where $i$ and $j$ are adjacent with respect to the cyclic ordering. 
 \end{lem}

 \subsection{$3$--manifold topology} \label{subsec:3mantop}
 
 To apply the techniques of Kapovich--Kleiner to groups in $\cC$, we will need the following proposition. 
 
 \begin{prop} \label{prop:3-man_to_pd3}
   Suppose a finitely generated, one-ended group $G$ is the fundamental group of a $3$--manifold. Then $G$ acts properly and simplicially on a coarse $PD(3)$ space $X$. In particular, for all $x \in X$, the orbit map $G \rightarrow X$ given by $g \mapsto g \cdot x$ is uniformly proper. 
 \end{prop}

 A similar result is used implicitly in \cite[Section~8]{kapovichkleiner}, but was not stated explicitly. For the benefit of the reader, we give a detailed proof using standard techniques from the topology of $3$--manifolds. Before proving the proposition, we briefly recall some useful background on $3$--manifolds. For more detail, we refer the reader to any of \cite{hempel,kapovich,hatcher_3man}. 
 
 \begin{defn}
    A $3$--manifold is {\it aspherical} if it is connected and its universal cover is contractible. A $3$--manifold $M$ is {\it prime} if whenever $M$ can be written as connected sum, $M \cong P \# Q$, then either $P \cong S^3$ or $Q \cong S^3$. A $3$--manifold $M$ is {\it irreducible} if every $2$--sphere $S^2 \subset M$ bounds a ball $B^3 \subset M$.
If $M$ is a $3$--manifold with boundary, an aspherical boundary component $F$ of $M$ is {\it incompressible} if the natural map $\pi_1(F) \rightarrow \pi_1(M)$ is injective.
 \end{defn}

 
  \begin{proof}[Proof of Proposition~\ref{prop:3-man_to_pd3}]
   Suppose the finitely generated group $G$ is the fundamental group of a $3$--manifold $M$. By the Compact Core Theorem \cite{scott73b}, the group $G$ is the fundamental group of a compact manifold $N$ such that each boundary component $\pi_1$--injects into $N$.
We may assume, without loss of generality, that $N$ has aspherical boundary, since any $S^2$ or $\R P^2$ boundary component can be capped off by attaching a $3$--cell without changing the fundamental group.  Thus $N$ has aspherical, incompressible boundary.

Since $G=\pi_1(N)$ is one-ended, a standard argument shows that $N$ itself is aspherical as follows.  Let $\hat{N}\to N$ be an orientable cover of degree at most two.  Then the finite index subgroup $\hat{G} = \pi_1(\hat{N})$ is also one-ended.
By the Prime Decomposition Theorem, $\hat{N}$ is the connected sum $P_1 \# \ldots \# P_k$ of finitely many compact prime $3$--manifolds.  Since $\hat{G}$ is one-ended, it does not split as a nontrivial free product.  Thus all but one of the prime factors are simply connected and, hence, have empty boundary.
By the Poincar\'{e} Conjecture the closed, simply connected, prime factors are all spheres, so that $\hat{N}$ is prime.  Because $\hat{G}$ is one-ended, $\hat{N}$ must also be irreducible since the only other possibility is $S^2 \times S^1$, which has a $2$--ended fundamental group.
Since $\hat{N}$ is compact, connected, orientable, and irreducible, and $\hat{G}$ is infinite, the Sphere Theorem implies that $\hat{N}$ is aspherical (see for example 
Corollary~3.9 of \cite{hatcher_3man}).
Since $N$ and $\hat{N}$ have the same contractible universal cover, we see that $N$ is a compact, aspherical $3$--manifold with aspherical, incompressible boundary components.

To conclude the proof, we will see that $G$ acts properly and simplicially on a coarse $PD(3)$ space.
Let $DN$ denote the double of $N$ across its boundary.  Since $N$ and its boundary are compact and aspherical, $DN$ is a closed aspherical $3$--manifold.
Using any finite triangulation of $DN$, Lemma~\ref{lem:char_pd3} implies that its universal cover is a coarse $PD(3)$ space on which $DG = \pi_1(DN)$ acts simplicially.
Since each boundary component of $N$ $\pi_1$--injects into $N$, the double $DG$ is the fundamental group of a graph of groups with one vertex for each copy of $N$ and one edge for each component of $\boundary N$.
In particular, $\pi_1(N) \to \pi_1(DN)$ is injective.
Thus the subgroup $G < DG$ acts properly and simplicially on the universal cover of $DN$.
    \end{proof}

 \subsection{Classification theorems} \label{3manproof}
  
  \begin{thm}
  \label{thm:3manclassification}
  Let $G \cong \pi_1(S_g) *_{\la a^m=b^n \ra} \pi_1(S_h) \in \mathcal{C}_{m,n}$, where $a\in \pi_1(S_g)$ and $b \in \pi_1(S_h)$ are homotopy classes of essential simple closed curves. Then $G$ is the fundamental group of a $3$-manifold if and only if one of the following holds: 
  \begin{enumerate}
   \item $m=n=1$;
   \item $m=1$, $n=2$, and $b$ is the homotopy class of a non-separating curve; or,
   \item $m=n=2$, and $a$ and $b$ are homotopy classes of non-separating curves. 
  \end{enumerate}
 \end{thm}
 
 \begin{proof}
  Let $G \cong \pi_1(S_g) *_{\la a^m=b^n \ra} \pi_1(S_h) \in \mathcal{C}_{m,n}$, where $a\in \pi_1(S_g)$ and $b \in \pi_1(S_h)$ are homotopy classes of essential simple closed curves. 
  
    \begin{center}
   {\it Construction of $3$--manifold structure.}
  \end{center}

  Suppose first that conditions (1), (2), or (3) hold. We will realize $G \cong \pi_1(S_g) *_{\la a^m=b^n \ra} \pi_1(S_h)$ as the fundamental group of $M_G$, an aspherical $3$-manifold with boundary.

  Suppose $a_0$ is an essential simple closed curve on $S_g$ in the homotopy class of $a \in \pi_1(S_g)$, and suppose $b_0$ is an essential simple closed curve on $S_h$ in the homotopy class of $b \in \pi_1(S_h)$. To construct $M_G$, we will glue an $I$--bundle over $S_g$ to an $I$--bundle over $S_h$ along an annulus in the boundary of each bundle, where the annulus on the $I$--bundle over $S_g$ is freely homotopic to $a_0^m$ and the annulus on the $I$--bundle over $S_h$ is freely homotopic to $b_0^n$. The choice of $I$--bundle over $S_g$ depends on whether $m=1$ or $m=2$; likewise, the choice of $I$--bundle over $S_h$ depends on whether $n=1$ or $n=2$. We describe the $I$--bundle over $S_g$ and the choice of annulus on its boundary; the construction is analogous for $S_h$. 
  
      \begin{figure}
      \begin{overpic}[scale=.25,  tics=5]{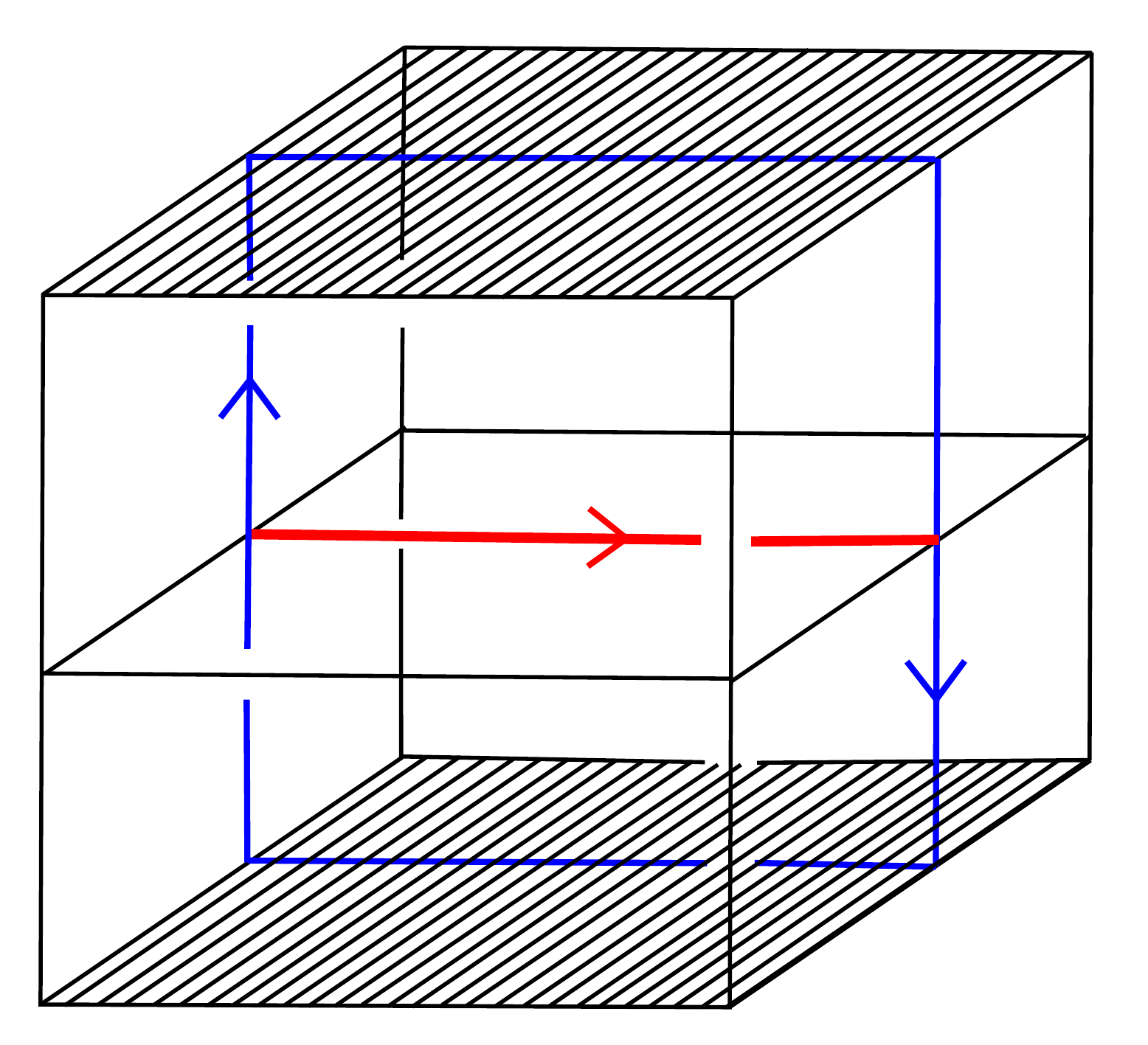}
      \put(40,47){$a_0$}
      \end{overpic}
	\caption{{\small Pictured above is the solid Klein bottle, formed by the $3$-cube $[-1,1]^3$ modulo the identification $(-1,y,z) \sim (1, y, -z)$: the left-hand face of the cube has been identified to the right-hand face of the cube by a reflection about the horizontal mid-plane. The top and bottom shaded faces glue to form an annulus. The solid Klein bottle is a subspace of an $I$--bundle over the surface that is twisted over the a non-separating $a_0$.    }}
      \label{Ibundle}
     \end{figure}
  
  If $m=1$, let $M_g \cong S_g \times [-1,1]$ be the trivial $I$--bundle over $S_g$. Let $N(a_0) \subset S_g$ be a regular neighborhood of $a_0$, and let $A_g = N(a_0) \times \{1\} \subset S_g \times I$ be an annulus on the boundary of $M_g$. Then $A_g$ is freely homotopic to the curve $a_0$. 
  
  If $m=2$, then by assumption, the homotopy class $a$ is represented by a non-separating simple closed curve $a_0$.
In this case, there exists a retraction $r\colon S_g \to \Image(a_0)$,
and the induced map $r_*$ gives rise to a homomorphism $\pi_1(S_g) \to \Z \to \Z/2\Z$
sending $a$ to a nontrivial element.
Any homomorphism $\phi\colon \pi_1(S_g) \rightarrow \Z/2\Z$ gives rise to an $I$--bundle over $S_g$ whose monodromy---i.e., the twisting over various curves---is given by $\phi$.
We briefly recall the construction of this bundle.
Let $\Z/2\Z$ act on the interval $I=[-1,1]$ via multiplication by $\pm 1$.
Then $\phi$ induces a corresponding action of $\pi_1(S_g)$ on $I$.
Let $\pi_1(S_g)$ act on $\tilde{S_g} \times I$ by a diagonal action that consists of the covering space action on the first factor and $\phi$ on the second factor.
The quotient by this action is an $I$--bundle over $S_g$ that is twisted over $a_0$:
the bundle over $a_0$ is a M\"{o}bius band with center circle $a_0$. Let $N(a_0)$ be a regular neighborhood of $a_0$ on $S_g$, and let $K_g \subset M_g$ be the union of the $I$--fibers over $N(a_0)$. Then the subspace $K_g \subset M_g$ forms a solid Klein bottle, as shown in Figure~\ref{Ibundle}. Let $A_g$ be the boundary of this subspace $K_g \subset M_g$, which is an annulus freely homotopic to $a_0^2$ and is shaded in Figure~\ref{Ibundle}. 

Let $M_h$ be the $I$--bundle over $S_h$ formed analogously, depending on whether $n=1$ or $n=2$. Let $A_h$ be the corresponding annulus on the boundary of $M_h$ that is freely homotopic to $b_0^n$. Identify $A_g\subset M_g$ and $A_h \subset M_h$ via a homeomorphism to form the aspherical space $M_G$ as desired. Thus, if (1), (2), or (3) hold, $G$ is the fundamental group of an aspherical $3$--manifold.

  \begin{center}
    {\it Obstruction to the existence of a $3$--manifold structure.}
  \end{center}
  
  Suppose now that none of (1)--(3) hold. Suppose towards a contradiction that $G$ is the fundamental group of a $3$-manifold. Then, by Proposition~\ref{prop:3-man_to_pd3}, $G$ acts properly on a coarse $PD(3)$ space~$X$. As described below, $G$ also acts freely and cocompactly on a model space $Y$, and suppose $W$ is a convex subset of $Y$. Then, the orbit maps $G \rightarrow X$ and $G \rightarrow Y$ given by $g \mapsto g \cdot x$ and $g \mapsto g \cdot y$ for $x \in X$ and $y \in W \subset Y$ define a $G$--equivariant uniformly proper map $f\colon Y\rightarrow X$,
which restricts to a $G$--equivariant uniformly proper map
$f\colon W \rightarrow X$.
We will examine the action of $G$ on a model space $Y$ and apply the coarse separation results of Kapovich--Kleiner to conclude $G$ is not the fundamental group of a $3$--manifold. 
  
  \begin{center}
   {\it Action of $G$ on a model space $Y$.}
  \end{center}

  As shown in Construction~\ref{surfmodel}, $G$ acts freely and cocompactly by isometries on a model space $Y$ of type $(m,n,\mathbb{X}, 2)$ as defined in Construction~\ref{modelspaces}. The space $Y$ is a locally-finite cell complex by construction, and the cells of $Y$ may be subdivided $G$--equivariantly to give $Y$ the structure of a bounded geometry metric simplicial complex. The subgroup $\la a,b \ra \leq G$ acts geometrically on one copy of $T_{m,n} \times \R$ in $Y$. The action of $a$ and $b$ on $T_{m,n} \times \R$ decomposes as a product. The element $a$ cyclically permutes the $m$ edges adjacent to a vertex $v$ of valance $m$ in $T_{m,n} \times \{0\}$ and translates by $n$ units in the $\R$--direction. Likewise, the element $b$ cyclically permutes the $n$ edges adjacent to a vertex $w$ of valance $n$ adjacent to $v$ in $T_{m,n} \times \{0\}$ and translates by $m$ units in the $\R$--direction. Note that if $m=1$, then $a$ acts only by translation in the $\R$--direction, and if $n=1$, then $b$ acts only by translation in the $\R$--direction. Recall, each line $v \times \R \subset T_{m,n} \times \R$ is called an {\it essential line}. 
  
  \begin{center}
   {\it Case 1: $m \geq 2$, $n \geq 3$.}
  \end{center}
  
  Suppose $m \geq 2$, $n\geq 3$. Suppose $\la a,b \ra \leq G$ stabilizes $T_{m,n} \times \R \subset Y$ as described above. Choose a geodesic ray $\rho$ based at the vertex $w$ in $T_{m,n} \times \{0\}$, and let $P_0 = \rho \times \R \subset T_{m,n} \times \R$. The space $P_0$ is homeomorphic to $\R_2^{+}$. Let $P_i = b^iP_0$ for $0 \leq i \leq n-1$. Then $P_i \cap P_j = \{w\} \times \R$, a line we denote by $\ell$. Glued to $\ell$ is also a copy of $\widetilde{S_h}$, the universal cover of the surface $S_h$. Let $H_1$ and $H_2$ be the two half-planes of $\widetilde{S_h}$ bounded by $\ell$. Let $W$ be the union of $P_i$ for $0 \leq i \leq n-1$,  $H_1$, and $H_2$ along $\ell$.  Then $W$ is union of $n+2$ half-planes and satisfies the conditions of Lemma~\ref{separation}. 
  
   By Lemma~\ref{separation}, there is a unique cyclic ordering on the index set of the half-planes in $f(W)$, and this cyclic order is preserved by any homeomorphism of $X$ that preserves $f(W)$. The group of permutations of $n+2$ elements which preserves a cyclic order is isomorphic to the dihedral group $D_{n+2}$ of order $2(n+2)$. 
   Since $b \in G$ preserves $W$ and $f$ is $G$-equivariant, $b$ preserves $f(W)$, and thus $b$ preserves the cyclic order on the index set of the half-planes in $f(W)$. Therefore, $b$ corresponds to an element $\sigma \in D_{n+2}$. However, the element $b$ acts on $W$ by cyclically permuting the half-planes $P_i$ and stabilizing each half-plane $H_i$, $i = 1,2$. Thus, $\sigma \in D_{n+2}$ is an $n$--cycle. When $n \geq 3$, $D_{n+2}$ does not contain an $n$--cycle, a contradiction. Thus, we conclude that  $G$ is not the fundamental group of a $3$--manifold. 
   
   \begin{center}
    {\it Case 2: $m=1$, $n \geq 3$.}
   \end{center}
  
  Suppose $m=1$ and $n \geq 3$. The arguments in this case are similar to those in Case 1 and account for the fact that $T_{1,n}$ is finite, so no geodesic ray $\rho \subset T_{1,n}$ exists. 
  
  Let $T_{1,n} \times \R \subset Y$ be the complex stabilized by $\la a, b \ra$. Let $\ell_0$ be the essential line in $T_{1,n} \times \R$ stabilized by $b$, and let $\ell_1, \ldots, \ell_n$ be the other $n$ essential lines in $T_{1,n} \times \R$. In the model space $Y$, the line $\ell_0$ is glued to a copy of $\widetilde{S_h}$, the universal cover of $S_h$, along a line. Let $H_1$ and $H_2$ be the two half-planes of $\widetilde{S_h}$ bounded by $\ell_0$. Similarly, each line $\ell_i$ for $1 \leq i \leq n$ is identified to a copy of $\widetilde{S_g}$, the universal cover of $S_g$, along a line. Let $J_{i1}$ and $J_{i2}$ be the two half-planes of $\widetilde{S_g}$ bounded by $\ell_i$. Finally, let $e_1, \ldots, e_n$ be the $n$ edges in $T_{1,n}$ so that $e_i \cap \ell_i\neq \emptyset$, and let $P_{ij}$ be the union of $e_i \times \R$ with $J_{ij}$ for $1 \leq i \leq n$ and $j=1,2$. Let $W$ be the union of $P_{ij}$ for $1\leq i \leq n$, $j=1,2$, $H_1$, and $H_2$ along $\ell_0$. Then $W$ is the union of $2n+2$ half-planes and satisfies the conditions of Lemma~\ref{separation}. 
   
  By Lemma~\ref{separation}, there is a unique cyclic ordering on the half-planes in $f(W)$ that is preserved by the element $b$ since $b$ preserves $W$. Thus, $b$ corresponds to an element $\sigma \in D_{2n+2}$. The element $b$ stabilizes each half-plane $H_1$ and $H_2$, so $b$ is either trivial or a reflection. However, $b$ cyclically permutes the lines $\ell_i$ for $1 \leq i \leq n$, and $n \geq 3$, so the order of $\sigma$ is greater than two, a contradiction. Thus, in this case, $G$ is not the fundamental group of a $3$--manifold. 
  
    \begin{center}
   {\it Case 3: $m=n=2$ and $a$ or $b$ is separating.}
  \end{center}

  Suppose $m=n=2$ and $b$ is a separating curve. The argument is analogous if $a$ is a separating curve. 
  
  Note first that if $b$ is the homotopy class of a separating curve, then $b$ is an element of the commutator subgroup of $\pi_1(S_h)$. So, $b$ maps to the identity in any homomorphism from $\pi_1(S_h)$ to an abelian group. In particular, if $\phi\colon\pi_1(S_h) \rightarrow \Z / 2\Z$ is a homomorphism then $\phi(b) = 0$. 
  
  A homomorphism $\pi_1(S_h) \rightarrow \Z / 2\Z$ is given as follows. Let $\ell \subset Y$ be the line stabilized by $\la b \ra$. The line $\ell$ is glued to $\widetilde{S_h}$, the universal cover of $S_h$; let $H_1$ and $H_2$ be the two half-planes in $\widetilde{S_h}$ bounded by $\ell$. Let $W = H_1 \cup H_2$, which satisfies the conditions of Lemma~\ref{separation}. The $G$--equivariant uniformly proper map $f\colon Y \rightarrow X$ described above restricts to a $G$--equivariant uniformly proper map $f\colon W\rightarrow X$. By Lemma~\ref{separation}, $f(W)$ coarsely separates $X$ into two deep components. The subgroup $\pi_1(S_h)$ stabilizes $W$ and hence $f(W)$, which yields a homomorphism $\phi\colon\pi_1(S_h) \rightarrow \Z/2\Z$. 
  
  To conclude this case, we show that $\phi(b) \neq 0$ by applying Lemma~\ref{separation} a second time. The line $\ell$ is also incident to a copy of $T_{2,2} \times \R \equiv \E^2$. Let $E_1$ and $E_2$ be the two half-planes in $T_{2,2} \times \R$ bounded by $\ell$. Let $W' = H_1 \cup H_2 \cup E_1 \cup E_2$, a union of four half-planes that satisfies the  conditions of Lemma~\ref{separation}. 
  The $G$--equivariant uniformly proper map $Y \rightarrow X$ restricts to a $G$--equivariant uniformly proper map $f'\colon W' \rightarrow X$ that extends the uniformly proper map $f\colon W\rightarrow X$.
  By the arguments of Case (1), the element $b \in G$ corresponds to a transposition $\sigma \in D_4$. So, the half-planes $f'(E_1)$ and $f'(E_2)$ must be opposite in the cyclic order. Thus, $f'(E_1)$ and $f'(E_2)$ are in different deep components of $X \backslash f'(H_1 \cup H_2) = X \backslash f(H_1 \cup H_2)$. (This fact, while intuitive, can be more carefully justified using \cite{hruskastark}.) Therefore, $b$ non-trivially permutes the deep components of $X \backslash f(H_1 \cup H_2)$, so $\phi(b) \neq 0$, a contradiction. Thus, in this case, $G$ is not the fundamental group of a $3$--manifold. 
  
  \begin{center}
   {\it Case 4: $m=1$, $n=2$, and $b$ is separating.}
  \end{center}

The modification we used above to pass from Case~1 to Case~2 can also be used in a similar way to modify our argument from Case~3 into a proof for Case~4.

 Cases 1--4 cover all possibilities that conditions (1)--(3) of the theorem do not hold. Thus, if conditions (1)--(3) do not hold, the group $G$ is not the fundamental group of a $3$-manifold. 
\end{proof}

The next theorem implies that each right-angled Coxeter group with generalized $\Theta$--graph nerve acts properly on a contractible $3$--manifold, since every generalized $\Theta$--graph is planar. The theorem is due to Davis--Okun, and was used implicitly in the proof of \cite[Theorem~11.4.1]{davisokun}. We give an explicit proof here for the benefit of the reader. 
 
\begin{thm}[Davis--Okun]
\label{thm:racg_3man}
Let $W$ be a right-angled Coxeter group with nerve a flag simplicial complex $L$. Suppose $L$ is planar and connected with dimension at most $2$. Then the group~$W$ acts properly on a contractible $3$--manifold. 
\end{thm}
 
\begin{proof}
Embed the complex $L$ into $S^2$, the $2$--sphere. Fill each complementary region on the sphere with $2$--simplices by adding a vertex in the interior of each region and coning off the boundary of the region to the new vertex. This procedure produces a flag triangulation $S$ of $S^2$, which has $L$ as a full subcomplex. Let $W_S$ be the right-angled Coxeter group with nerve the flag triangulation $S$. Since the nerve $S$ is a $2$--sphere, the Davis complex $\Sigma_S$ associated to $S$ is a contractible $3$-manifold on which $W_S$ acts properly and cocompactly. The original group $W = W_L$ is a subgroup of $W_S$ so it also acts properly on the same $3$--manifold.
\end{proof}

  Using Proposition~\ref{3mancovers} and Theorem~\ref{thm:racg_3man}, we recover an alternate proof of Theorem~\ref{thm:surf_amal_3mancover}.
 
  \begin{cor} \label{cor:surf_3man}
  Let $G \in \mathcal{C}_{m,n}$. Then $G$ has a finite-index subgroup that acts freely on a contractible $3$--manifold. 
 \end{cor}


\appendix
\section{Quasi-isometry classification}
\label{sec:QIClassification}
\subsection{Construction of quasi-isometries}
\label{sec:QIConstruction}

In this appendix we construct quasi-isometries between model spaces, and we use the following results of Behrstock--Neumann \cite{behrstockneumann} and Whyte \cite{whyte} on the geometry of the vertex spaces of the model spaces. 

\begin{thm}[\cite{behrstockneumann}, Theorem~1.2]
\label{bs1}
Let $X$ be a fattened tree with a chosen boundary component $\partial_0X$. Then there exist $K$ and a function $\phi\colon\RR \rightarrow \RR$ such that for any $K'$ and any $K'$--bilipschitz homeomorphism $\Phi_0$ from $\partial_0X$ to a boundary component $\partial_0X_0$, the map $\Phi_0$ extends to a $\phi(K')$--bilipschitz homeomorphism $\Phi\colon X \rightarrow X_0$ which is $K$--bilipschitz on every other boundary component. 
\end{thm}

\begin{thm}[\cite{whyte}, Theorems 1.1 and~5.1]
\label{k3}
Let $m_1, m_2 \geq 2$ and $n_1, n_2 \geq 3$ be integers. Then there is a quasi-isometry from $T_{m_1, n_1}$ to $T_{m_2, n_2}$ mapping vertices of $T_{m_1, n_1}$ bijectively onto vertices of $T_{m_2, n_2}$. 
\end{thm}

\begin{cor}
\label{treeR}
Let $m_1, m_2 \geq 2$ and $n_1, n_2 \geq 3$ be integers and let $K \geq 0$. Every $K$--bilipschitz map from an essential line of $T_{m_1, n_1} \times \RR$ to an essential line of $T_{m_2, n_2} \times \RR$ extends to a quasi-isometry from $T_{m_1, n_1} \times \RR$ to $T_{m_2, n_2} \times \RR$ mapping essential lines of $T_{m_1, n_1} \times \RR$ bijectively onto essential lines of $T_{m_2, n_2} \times \RR$ via $K$--bilipschitz maps. 
\end{cor}

 We use the next proposition to piece together quasi-isometries between the vertex spaces in two model spaces. 

\begin{prop}
\label{gluing}
 Suppose $(X,d_1)$ and $(Y,d_2)$ are model spaces of the type given in Construction~\ref{modelspaces}. Suppose $\Psi\colon X \rightarrow Y$ is a map that induces an isomorphism between the underlying trees of $X$ and $Y$; $\Psi$ restricts to a $(K,C)$--quasi-isometry on corresponding vertex spaces of type $T_{m,n} \times \R$; $\Psi$ restricts to a $K$--bilipschitz homeomorphism between corresponding vertex spaces of type $X_i$; and, $\Psi$ restricts to a $K$--bilipschitz homeomorphism between corresponding edge spaces. Then $\Psi$ is an $(L,D)$--quasi-isometry.  
\end{prop}

We caution the reader that the model spaces here are constructed as a tree of spaces in which adjacent vertex spaces are directly glued together.  Additionally, some of the data in Proposition~\ref{gluing} consists of quasi-isometries of vertex spaces, which do not usually behave as well as bilipschitz maps with respect to gluing operations.  However, the given maps are more tightly controlled bilipschitz maps on the fattened tree vertex spaces, where there is a uniform bound on the distance between incident edge spaces.  The bipartite structure of the underlying tree, over which alternate quasi-isometries with bilipschitz maps, allows one to paste together the maps to produce a global quasi-isometry.
The proof of the proposition is quite similar to the proof of \cite[Proposition~2.16]{cashenmartin}. We leave the details as a routine exercise for the reader.

  Model spaces of type $(0,0,\XX,s)$ and $(1,n,\XX,s)$ are $\delta$--hyperbolic for all $n,s \geq 1$.  Malone and Dani--Thomas have shown that the pair of values $s$ and $n$ are a complete quasi-isometry invariant in the hyperbolic case (see \cite[Theorem~4.14]{malone} and \cite[Theorem~1.5]{danithomas}).

 \begin{thm}\cite{malone} \cite{danithomas}
 \label{hypcase}
  Let $s,s', n, n'$ be integers. Model spaces of type $(0,0,\XX,s)$ and $(0,0,\XX',s')$ are quasi-isometric if and only if $s = s'$. Model spaces of type $(1,n,\XX,s)$ and $(1,n',\XX',s')$ are quasi-isometric if and only if $n=n'$ and $s=s'$. 
 \end{thm}

  For $m>1$, a model space of type $(m,n,\XX,s)$ has isometrically embedded flat planes in the $T_{m,n} \times \R$ pieces, and is therefore not $\delta$--hyperbolic. The following theorem proves there are at most three quasi-isometry classes among the non-hyperbolic model spaces. Using the action of groups on these spaces, the three quasi-isometry types are distinguished in Theorem~\ref{racgQI}. 

\begin{thm}[Quasi-isometries between model spaces]
\label{qis}
Let $\XX$ be a finite set of fattened trees. 
\leavevmode
  \begin{enumerate}
  \renewcommand{\theenumi}{\Roman{enumi}}
  \item
  \label{item:qis:isolate}
  For all integers $m \geq 2$, $n \geq 3$, and $s \geq 1$, a model space of type $(m,n,\XX,s)$ is quasi-isometric to the standard model space of type $(3,3,1)$. 
  \item
  \label{item:qis:mixed}
  For all $s \geq 1$, a model space of type $(2,2,\XX,s)$ is quasi-isometric to the standard model space of type $(2,2,1)$.
  \item
  \label{item:qis:Seifert}
  For all integers $m \geq 2$ and $n \geq 3$, the model space of type $(m,n,0)$ is quasi-isometric to the model space of type $(2,3,0)$. 
 \end{enumerate}
\end{thm}

\begin{rem}
In the case of surface amalgams, the parameter $s$ is always equal to $2$.  Thus we do not need to change $s$ in order to get quasi-isometries between their model spaces.
In this special case, the proof of Theorem~\ref{qis} follows easily by combining Theorem~\ref{bs1} and Corollary~\ref{treeR}.
However, in order to construct quasi-isometries between right-angled Coxeter groups,
we must deal with the more subtle effect of changing the value of $s$.
\end{rem}

\begin{proof} Let $\XX$ be a finite set of fattened trees. 
To prove (\ref{item:qis:isolate}), we show that each model space $Y$ of type $(m,n,\XX,s)$ is quasi-isometric to the standard model space $Y'$ of type $(s+2, s+2, s)$, and that $Y'$ is quasi-isometric to the standard model space $Y_0$ of type $(3,3,1)$.
 
 Recall that $X_0$ denotes the standard fattened tree. Since $\XX$ is a finite set, there exist constants $K_0, K_1 \geq 1$ so that any $K_0$--bilipschitz homeomorphism from a boundary component of $X_i$ to a boundary component of $X_0$ extends to an $K_1$--bilipschitz homeomorphism from $X_i$ to $X_0$ that is $K_0$--bilipschitz on every other boundary component by Theorem~\ref{bs1}.  By Corollary~\ref{treeR}, there exist constants $K_2 \geq 1$ and $C \geq 0$ so that every $K_2$--bilipschitz homeomorphism from an essential line of $T_{m,n} \times \RR$ to $T_{s+2,s+2} \times \RR$ extends to an $(K_2, C)$--quasi-isometry from $T_{m,n} \times \RR$ to $T_{s+2,s+2} \times \RR$. Let $K = \max\{K_0, K_1, K_2\}$.

 Define a map $\Psi$ from $Y$ to $Y'$ recursively as follows. Map one piece $T_{m,n} \times \RR$ of $Y$ to one piece $T_{s+2,s+2} \times \RR$ of $Y'$ by a $(K,C)$--quasi-isometry that is $K_0$--bilipschitz on each essential line of $T_{m,n} \times \RR$. Along each essential line of $T_{m,n} \times \RR$ the $K_0$--bilipschitz map can be extended to a $K$--bilipschitz map from the copy of $X_i$ incident to this line to the $i^{th}$ copy of $X_0$ incident to the image of the line in $T_{s+2,s+2} \times \R$ that is $K_0$--bilipschitz on every boundary component of $X_i$ for each $1 \leq i \leq s$. Similarly, along each other boundary line of the $X_i$ incident to an essential line of the original copy of $T_{m,n} \times \RR$, the map can be extended over its neighboring copies of $X_i$ and $T_{m,n} \times \R$. Continuing in this way exhausts all vertex spaces in the geometric graph of spaces and defines $\Psi\colon Y \rightarrow Y'$. Since $Y$, $Y'$, and $\Psi$ satisfy the conditions of Proposition~\ref{gluing}, it follows that $\Psi$ is a quasi-isometry.

  We now show $Y'$ is quasi-isometric to $Y_0$. First, define a map $\phi\colon T_{3,3} \times \R \rightarrow T_{s+2,s+2} \times \R$ as follows. Choose a base vertex $v$ in $T_{3,3}$ and a path of length $s-1$ based at $v$. Along each vertex $w$ of the path, construct a new path of length $s-1$ based at each vertex at distance $1$ from $w$ that is not in the original path. Continue this recursively, to select a set of paths of length $s-1$ in $T_{3,3}$, and collapse each path to a point to obtain the regular tree $T_{s+2, s+2}$ as pictured in Figure~\ref{figure:T_{2,2}}. Extend this map to a $(s,s)$--quasi-isometry $\Phi\colon T_{3,3} \times \R \rightarrow T_{s+2, s+2} \times \R$ so that the map is an isometry on the $\R$ factor. Then, $\phi$ maps strips isometric to $[0,s-1] \times \R$ onto the $\R$ factor, where the strips are at distance one from each other, and $s$ essential lines are identified under $\phi$. Define $\Phi\colon Y' \rightarrow Y_0$ by applying $\phi$ to each piece of type $T_{3,3} \times R$ and the identity on each piece of type $X_i$. Since $Y'$, $Y_0$, and $\Phi$ satisfy the conditions of Proposition~\ref{gluing}, the map $\Phi$ is a quasi-isometry. Therefore, $Y$ is quasi-isometric to $Y_0$, proving (\ref{item:qis:isolate}).
 
   \begin{figure} 
 \centering
 \includegraphics[scale=.8]{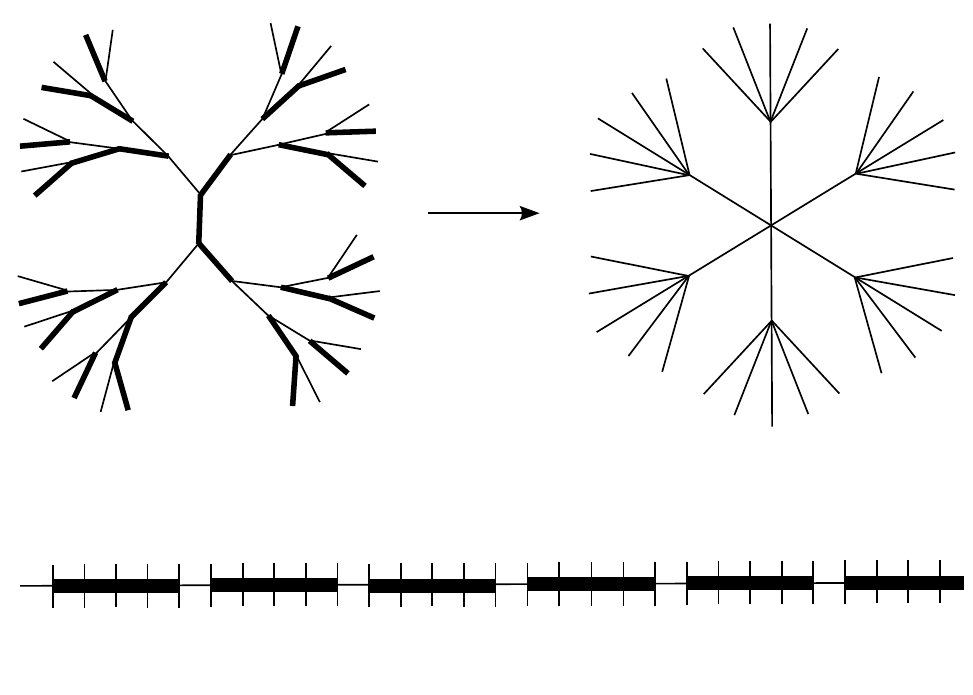}
 \caption{{\small In the top trees, a quasi-isometry from $T_{3,3}$ to $T_{6,6}$ is given by collapsing paths of length three, as drawn in thick lines.  Below, collapsing each path drawn in thick lines defines a quasi-isometry $T_{2,2}$ to $T_{2,2}$.}}
 \label{figure:T_{2,2}}
 \end{figure}

 To prove (\ref{item:qis:mixed}), let $Z$ be a model space of type $(2,2,\XX,s)$, let $Z'$ be the standard model space of type $(2,2,s)$, and let $Z_0$ be the standard model space of type $(2,2,1)$.  It follows immediately from the above argument that $Z$ is quasi-isometric to $Z'$.

 To complete the proof of Theorem~\ref{qis}, we must prove $Z'$ is quasi-isometric to $Z_0$. We first define a quasi-isometry $\phi$ from $T_{2,2} \times \R$ to $T_{2,2} \times \R$ as follows. Identify $T_{2,2,}$ with $\R$ via an isometry mapping vertices of $T_{2,2}$ onto $\Z$. Map $T_{2,2}$ to $T_{2,2}$ by sending the interval $\bigl[ ks, (k+1)s-1 \bigr]$ to the point $k$ and the interval $\bigl[ (k+1)s-1, (k+1)s \bigr]$ to the interval $[k, k+1]$ as pictured in Figure~\ref{figure:T_{2,2}}. Extend this map to a map $\phi\colon T_{2,2} \times \R \rightarrow T_{2,2} \times \R$ so that the map is an isometry on the $\R$ factor. Then the map $\phi$, which collapses parallel strips isometric to $[0, s-1] \times \R$ onto the $\R$ factor, where the strips are at distance one from each other, is an $(s,s)$--quasi-isometry. Moreover, $s$ essential lines are identified under this map. Thus, we may define  $\Phi\colon Z' \rightarrow Z_0$ by applying $\phi$ to each piece in $Z'$ of type $T_{2,2} \times \R$. Since $Z'$, $Z_0$, and $\Phi$ satisfy the conditions of Proposition~\ref{gluing}, the map $\Phi$ is a quasi-isometry. Therefore, $Z$ is quasi-isometric to $Z_0$ concluding the proof of (\ref{item:qis:mixed}).
 
The proof of (\ref{item:qis:Seifert}) follows immediately from Corollary~\ref{treeR}. 
 \end{proof}

\subsection{Quasi-isometry classification of right-angled Coxeter groups}
\label{sec:QIRACG}

The right-angled Coxeter group $W_{\Theta}$ is $\delta$--hyperbolic if and only if the linear degree of $\Theta$ is less than or equal to one.  The complete quasi-isometry classification in this setting is given by Dani--Thomas. We state their theorem here for completeness. 

\begin{thm} \cite[Theorem~1.5]{danithomas}
 Let $\Theta = \Theta(n_1, \ldots, n_k)$ and $\Theta' = \Theta(n_1', \ldots, n_k')$ be generalized $\Theta$--graphs with $n_i, n_i' \geq 2$ for $i \geq 2$, and assume $n_1 \leq n_1'$. The groups $W_{\Theta}$ and $W_{\Theta'}$ are quasi-isometric if and only if one of the following three conditions holds:
 \begin{enumerate}
  \item $n_1 = n_1' =1$ and $k=k'$; or,
  \item $n_1 = 1$, $n_1' \geq 2$ and $k' = 2(k-1)$; or,
  \item $n_1 \geq 2$, $n_1' \geq 2$ and $k=k'$. 
  \end{enumerate}
\end{thm}

In particular, it follows from the previous theorem that there are infinitely many quasi-isometry classes of hyperbolic right-angled Coxeter groups with generalized $\Theta$--graph nerves.
We prove that in the non-hyperbolic case there are exactly three quasi-isometry classes. To distinguish the three quasi-isometry classes, we use a result of Caprace \cite{caprace,caprace-erra} that implies $W_{\Theta}$ is hyperbolic relative to the subgroup generated by the linear part of $\Theta$.
The proof then follows by a similar strategy as Kapovich--Leeb's proof that nongeometric $3$--manifolds with different types of geometric pieces cannot be quasi-isometric \cite{kapovichleeb}.

\begin{thm}[Quasi-isometric classification of non-hyperbolic right-angled Coxeter groups with generalized $\Theta$-graph nerve]
\label{racgQI} 
Let $\Theta$ and $\Theta'$ be generalized $\Theta$--graphs with linear degree $\ell \geq 2$ and $\ell' \geq 2$, respectively, and hyperbolic degree $h \geq 0$ and $h'\geq 0$, respectively.  Then the right-angled Coxeter groups $W_{\Theta}$ and $W_{\Theta'}$ are quasi-isometric if and only if one of the following three conditions holds:
  \begin{enumerate}
   \item $\ell = \ell' = 2$ and $h,h'\geq 1$; or,
   \item $\ell,\ell' \geq 3$ and $h,h'\geq 1$; or,
   \item $\ell,\ell' \geq 3$, and $h=h'=0$.
  \end{enumerate}
\end{thm}

 \begin{proof}
   Let $\Theta$ and $\Theta'$ be generalized $\Theta$--graphs  with linear degree $\ell \geq 2$ and $\ell'\geq 2$, respectively, and hyperbolic degree $h \geq 0$ and $h' \geq 0$, respectively.  By Construction~\ref{racgmodel}, $W_{\Theta}$ and $W_{\Theta'}$ are quasi-isometric to model spaces of type $(\ell, \ell,\XX, h)$ and $(\ell', \ell',\XX', h')$ respectively. Thus, by Theorem~\ref{qis}, if conditions (1), (2), or~(3) hold, $W_{\Theta}$ and $W_{\Theta'}$ are quasi-isometric. 
 
   Suppose conversely that $W_{\Theta}$ and $W_{\Theta'}$ are quasi-isometric. Let $\Theta_L \subset \Theta$ and $\Theta_{L}' \subset \Theta'$ denote the linear part of $\Theta$ and $\Theta'$, respectively. By  \cite[Theorem~A$'$]{caprace-erra}, the groups $W_{\Theta}$ and $W_{\Theta'}$ are hyperbolic relative to the right-angled Coxeter groups $W_{\Theta_L} \leq W_{\Theta}$ and $W_{\Theta_L'} \leq W_{\Theta'}$. 

Groups satisfying Condition (1) have a structure similar to $3$--manifolds whose geometric decomposition has only hyperbolic pieces.  On the other hand, groups satisfying (2) are similar to $3$--manifolds having both hyperbolic and Seifert fibered pieces.
Finally case (3) involves groups that resemble Seifert manifolds.
We proceed as in \cite{kapovichleeb} to show that the groups with different geometric types cannot be quasi-isometric to each other.

The group $W_{\Theta_L}$ acts properly discontinuously and cocompactly by isometries on $T_{\ell} \times \R$ where $T_{\ell}$ denotes a regular tree of valance $\ell$. If $\ell=2$ then every asymptotic cone of $W_{\Theta_L}$ is  homeomorphic to $\R^2$ and if $\ell>2$ every asymptotic cone of $W_{\Theta_L}$ is homeomorphic to $T \times \R$, where $T$ is a tree with nontrivial branching (see \cite[\S 3.3]{kapovichleeb}). In particular, neither asymptotic cone has a global cut-point. By \cite[Theorem~1.11]{drutusapir} such a group does not admit a proper relatively hyperbolic structure.  In fact, such a group cannot be quasi-isometric to any properly relatively hyperbolic group, since quasi-isometries induce homeomorphisms of asymptotic cones \cite[Proposition~3.12]{kapovichleeb}.
In particular, if $h=0$ then $h'=0$ and (3) holds.  We now assume that $h$ and $h'$ are nonzero.

To complete the proof, suppose by way of contradiction that $\ell=2$ and $\ell'\ge 3$
and suppose $\Phi\colon W_{\Theta'} \to W_\Theta$ is a quasi-isometry.  Then, by \cite[Theorem~1.7]{drutusapir}, $\Phi$ induces a quasi-isometric embedding of $W_{\Theta_L'}$ into $W_{\Theta_L}$, which induces a topological embedding of asymptotic cones $T \times \R \hookrightarrow\R^2$, where $T$ is a tree with nontrivial branching. Since such an embedding does not exist, we see that the combination $\ell=2$ and $\ell'\ge 3$ is impossible.  The only remaining possibility is that either (1) or (2) holds.
\end{proof}

 
\subsection{Quasi-isometry classification of surface group amalgams}
\label{sec:QISurfaceAmalgam}

\begin{thm}[Quasi-isometric classification of groups in $\mathcal{C}$]
\label{main1}
Let $G \in \mathcal{C}_{m,n}$ and $G' \in \mathcal{C}_{m',n'}$. Then $G$ and $G'$ are quasi-isometric if and only if one of the following conditions hold:
\begin{enumerate}
\item $m=m'=1$ and $n=n'$; or, 
\item $m=m'=n=n'=2$; or, 
\item $m\geq 2, n\geq 3, m'\geq 2$ and $n'\geq 3$. 
\end{enumerate}
\end{thm}

\begin{proof}
Let $G \in \mathcal{C}_{m,n}$ and $G' \in \mathcal{C}_{m',n'}$. The groups $G$ and $G'$ are hyperbolic if and only if $m=m'=1$. In this case, by \cite{malone}, $G$ and $G'$ are quasi-isometric if and only if $n=n'$. 

Suppose now that $m,m'\geq 2$. By Construction~\ref{surfmodel} and Theorem~\ref{qis}, if $m= n = 2$, $G$ is quasi-isometric to the standard model space of type $(2,2,1)$ and if $n>2$, $G$ is quasi-isometric to the standard model space of type $(3,3,1)$. Similarly, if $m' = n'=2$, $G'$ is quasi-isometric to the standard model space of type $(2,2,1)$ and if $n'>2$, $G'$ is quasi-isometric to the standard model space of type $(3,3,1)$. By Construction~\ref{racgmodel}, the standard model space of type $(k,k,1)$ is quasi-isometric to a right-angled Coxeter group with underlying graph a generalized $\Theta$--graph and linear degree $k$. Thus, if $m,m'\geq 2$, by Theorem~\ref{racgQI}, $G$ and $G'$ are quasi-isometric if and only if condition (2) holds or condition (3) holds.
\end{proof}

We conclude by remarking that each surface amalgam in $\mathcal{C}_{m,n}$ is quasi-isometric to a right-angled Coxeter group defined by a generalized $\Theta$--graph. 

\begin{cor}
\label{cor:main3}
Let $m, n \geq 1$, and let $G \in \mathcal{C}_{m,n}$. Then $G$ is quasi-isometric to one of the following right-angled Coxeter groups with nerve a generalized $\Theta$--graph:
\begin{enumerate}
\item If $m=1$, then $G$ is quasi-isometric to a right-angled Coxeter group $W_{\Theta}$ where $\Theta$ is a generalized $\Theta$--graph with linear degree $0$ and hyperbolic degree $2n+2$.
\item If $m=n=2$, then the group $G$ is quasi-isometric to a right-angled Coxeter group $W_{\Theta}$ where $\Theta$ is a generalized $\Theta$--graph with linear degree $2$ and hyperbolic degree at least~$1$.
\item If $m\geq 2$ and $n\geq 3$, then $G$ is quasi-isometric to a right-angled Coxeter group $W_{\Theta}$ where $\Theta$ is a generalized $\Theta$--graph with linear degree at least $3$ and hyperbolic degree at least~$1$.
\end{enumerate}
\end{cor}


\bibliographystyle{alpha}
\bibliography{refs}

\begin{thebibliography}{HPW16}

\bibitem[BKK02]{bestvinakapovichkleiner}
M.~Bestvina, M.~Kapovich, and B.~Kleiner.
\newblock Van {K}ampen's embedding obstruction for discrete groups.
\newblock {\em Invent. Math.}, 150(2):219--235, 2002.

\bibitem[BN08]{behrstockneumann}
J.A. Behrstock and W.D. Neumann.
\newblock Quasi-isometric classification of graph manifold groups.
\newblock {\em Duke Math. J.}, 141(2):217--240, 2008.

\bibitem[Cap09]{caprace}
P.-E. Caprace.
\newblock Buildings with isolated subspaces and relatively hyperbolic {C}oxeter
  groups.
\newblock {\em Innov. Incidence Geom.}, 10:15--31, 2009.

\bibitem[Cap15]{caprace-erra}
P.-E. Caprace.
\newblock Erratum to ``{B}uildings with isolated subspaces and relatively
  hyperbolic {C}oxeter groups''.
\newblock {\em Innov. Incidence Geom.}, 14:77--79, 2015.

\bibitem[CM17]{cashenmartin}
C.H. Cashen and A.~Martin.
\newblock Quasi-isometries between groups with two-ended splittings.
\newblock {\em Math. Proc. Cambridge Philos. Soc.}, 162(2):249--291, 2017.

\bibitem[CP08]{crisppaoluzzi}
J.~Crisp and L.~Paoluzzi.
\newblock Commensurability classification of a family of right-angled {C}oxeter
  groups.
\newblock {\em Proc. Amer. Math. Soc.}, 136(7):2343--2349, 2008.

\bibitem[CS94]{cullershalen}
M.~Culler and P.B. Shalen.
\newblock Volumes of hyperbolic {H}aken manifolds. {I}.
\newblock {\em Invent. Math.}, 118(2):285--329, 1994.

\bibitem[Dav08]{davis}
M.W. Davis.
\newblock {\em The geometry and topology of {C}oxeter groups}, volume~32 of
  {\em London Mathematical Society Monographs Series}.
\newblock Princeton University Press, Princeton, NJ, 2008.

\bibitem[DJ00]{DavisJanuszkiewicz00}
M.W. Davis and T.~Januszkiewicz.
\newblock Right-angled {A}rtin groups are commensurable with right-angled
  {C}oxeter groups.
\newblock {\em J. Pure Appl. Algebra}, 153(3):229--235, 2000.

\bibitem[DO01]{davisokun}
M.W. Davis and B.~Okun.
\newblock Vanishing theorems and conjectures for the {$\ell^2$}--homology of
  right-angled {C}oxeter groups.
\newblock {\em Geom. Topol.}, 5:7--74, 2001.

\bibitem[DS05]{drutusapir}
C.~Dru{\cb{t}}u and M.~Sapir.
\newblock Tree-graded spaces and asymptotic cones of groups.
\newblock {\em Topology}, 44(5):959--1058, 2005.
\newblock With an appendix by D. Osin and M. Sapir.

\bibitem[DST]{danistarkthomas}
P.~Dani, E.~Stark, and A.~Thomas.
\newblock Commensurability classification for certain right-angled {C}oxeter
  groups and geometric amalgams of free groups.
\newblock To appear in \emph{Groups Geom. Dyn.} arXiv:1610.06245.

\bibitem[DT17]{danithomas}
P.~Dani and A.~Thomas.
\newblock Bowditch's {JSJ} tree and the quasi-isometry classification of
  certain {C}oxeter groups.
\newblock {\em J. Topol.}, 10(4):1066--1106, 2017.

\bibitem[Hat02]{hatcher}
A.~Hatcher.
\newblock {\em Algebraic topology}.
\newblock Cambridge University Press, Cambridge, 2002.

\bibitem[Hat07]{hatcher_3man}
A.~Hatcher.
\newblock Notes on basic $3$--manifold topology.
\newblock \texttt{https://www.math.cornell.edu/$\sim$hatcher}, 2007.

\bibitem[Hem76]{hempel}
J.~Hempel.
\newblock {\em {$3$}--{M}anifolds}, volume~86 of {\em Ann. of Math. Studies}.
\newblock Princeton University Press, Princeton, NJ, 1976.

\bibitem[HPW16]{haissinskypaoluzziwalsh}
P.~Ha\"{i}ssinsky, L.~Paoluzzi, and G.~Walsh.
\newblock Boundaries of {K}leinian groups.
\newblock {\em Illinois J. Math.}, 60(1):353--364, 2016.

\bibitem[HST]{hruskastark}
G.C. Hruska, E.~Stark, and H.C. Tran.
\newblock Jordan cycling and virtually {K}leinian $3$--manifold groups.
\newblock In preparation.

\bibitem[HW99]{HsuWise99}
T.~Hsu and D.T. Wise.
\newblock On linear and residual properties of graph products.
\newblock {\em Michigan Math. J.}, 46(2):251--259, 1999.

\bibitem[Kap01]{kapovich}
M.~Kapovich.
\newblock {\em Hyperbolic manifolds and discrete groups}.
\newblock Modern Birkh\"auser Classics. Birkh\"auser Boston, Inc., Boston, MA,
  2001.

\bibitem[KK00]{kapovichkleiner}
M.~Kapovich and B.~Kleiner.
\newblock Hyperbolic groups with low-dimensional boundary.
\newblock {\em Ann. Sci. \'Ecole Norm. Sup. \textup{(}4\textup{)}},
  33(5):647--669, 2000.

\bibitem[KK05]{kapovichkleiner05}
M.~Kapovich and B.~Kleiner.
\newblock Coarse {A}lexander duality and duality groups.
\newblock {\em J. Differential Geom.}, 69(2):279--352, 2005.

\bibitem[KL95]{kapovichleeb}
M.~Kapovich and B.~Leeb.
\newblock On asymptotic cones and quasi-isometry classes of fundamental groups
  of {$3$}--manifolds.
\newblock {\em Geom. Funct. Anal.}, 5(3):582--603, 1995.

\bibitem[Mal10]{malone}
W.~Malone.
\newblock {\em Topics in geometric group theory}.
\newblock PhD thesis, The University of Utah, 2010.

\bibitem[Neu01]{neumann}
W.D. Neumann.
\newblock Immersed and virtually embedded {$\pi\sb 1$}--injective surfaces in
  graph manifolds.
\newblock {\em Algebr. Geom. Topol.}, 1:411--426, 2001.

\bibitem[Sco73]{scott73b}
G.P. Scott.
\newblock Finitely generated {$3$}--manifold groups are finitely presented.
\newblock {\em J. London Math. Soc. \textup{(}2\textup{)}}, 6:437--440, 1973.

\bibitem[Ser77]{serre}
J.-P. Serre.
\newblock {\em Arbres, amalgames, {${\rm SL}\sb{2}$}}, volume~46 of {\em
  Ast\'erisque}.
\newblock Soci\'et\'e Math\'ematique de France, Paris, 1977.
\newblock Written in collaboration with H. Bass.

\bibitem[Sta17]{stark}
E.~Stark.
\newblock Abstract commensurability and quasi-isometry classification of
  hyperbolic surface group amalgams.
\newblock {\em Geom. Dedicata}, 186:39--74, 2017.

\bibitem[SW79]{ScottWall79}
P.~Scott and T.~Wall.
\newblock Topological methods in group theory.
\newblock In {\em Homological group theory \textup{(}{D}urham, 1977\textup{)}},
  volume~36 of {\em London Math. Soc. Lecture Note Ser.}, pages 137--203.
  Cambridge Univ. Press, Cambridge, 1979.

\bibitem[Why99]{whyte}
K.~Whyte.
\newblock Amenability, bi-{L}ipschitz equivalence, and the von {N}eumann
  conjecture.
\newblock {\em Duke Math. J.}, 99(1):93--112, 1999.

\bibitem[Wis12]{wise}
D.T. Wise.
\newblock {\em From riches to {RAAGs}: $3$--manifolds, right-angled {A}rtin
  groups, and cubical geometry}, volume 117 of {\em CBMS Regional Conference
  Series in Mathematics}.
\newblock American Mathematical Society, Providence, RI, 2012.

\end{thebibliography}

\end{document}